\documentclass[14pt, draft]{article}
\usepackage[cp1251]{inputenc}
\usepackage[ukrainian, russian, english]{babel}

\usepackage{ amsfonts, amssymb}
\usepackage{amsmath}
\usepackage{amsthm}

\newtheorem{theorem}{Теорема}[section]

\newtheorem{lemma}[theorem]{Лема}

\theoremstyle{definition}

\begin{document}

\selectlanguage{ukrainian} \thispagestyle{empty}
 \pagestyle{myheadings}              

\pagestyle{myheadings}              

УДК 517.5 \vskip 3mm

\noindent \bf А.С. Сердюк  \rm (Інститут математики НАН України, Київ) \\
\noindent \bf Т.А. Степанюк  \rm (Інститут математики НАН України, Київ)

\noindent {\bf A.S. Serdyuk} (Institute of Mathematics NAS of
Ukraine, Kyiv) \\
 \noindent {\bf T.A. Stepaniuk} (Institute of Mathematics NAS of
Ukraine, Kyiv)\\

 \vskip 5mm

{\bf Наближення узагальнених інтегралів Пуассона інтерполяційними  тригонометричними поліномами}

\vskip 5mm

{\bf Approximation of generalized Poisson integrals by interpolation trigonometric polynomials}

\vskip 5mm

 \rm У даній роботі встановлено асимптотично непокращувані інтерполяційні аналоги нерівностей типу Лебега для $2\pi$--періодичних функцій $f$, які представляються узагальненими інтегралами Пуассона  функцій $\varphi$ з простору $L_p$, $1\leq p\leq \infty$. В зазначених нерівностях модулі відхилень  $|f(x)- \tilde{S}_{n-1}(f;x)|$ інтерполяційних поліномів Лагранжа при кожному $x\in\mathbb{R}$ оцінюються через найкращі наближення $E_{n}(\varphi)_{L_{p}}$ функцій $\varphi$ тригонометричними поліномами в $L_{p}$--метриках. Знайдено також асимптотичні рівності для точних верхніх  меж поточкових наближень інтерполяційними тригонометричними поліномами на класах $C^{\alpha,r}_{\beta,p}$ узагальнених інтегралів Пуассона функцій, що належать одиничним кулям просторів $L_p$, $1\leq p\leq\infty$.
 

\vskip 5mm

 \rm In this paper we establish asymptotically best possible interpolation Lebesgue--type inequalities for $2\pi$--periodic functions  $f$, which are representable as generalized Poisson integrals of the functions $\varphi$ from the space  $L_p$, $1\leq p\leq \infty$. In these inequalities the deviation of the interpolation Lagrange polynomials   $|f(x)- \tilde{S}_{n-1}(f;x)|$ for every $x\in\mathbb{R}$ is expressed via the best approximations  $E_{n}(\varphi)_{L_{p}}$ of the functions $\varphi$ be trigonometric polynomials in  $L_{p}$--metrics. We also find asymptotic equalities for the exact upper bounds of points approximations by interpolation trigonometric polynomials  on the classes  $C^{\alpha,r}_{\beta,p}$ of generalized Poisson integrals of the functions, which belong to the unit balls of the spaces $L_p$, $1\leq p\leq\infty$.
 
 \vskip 5mm

{\rm  \large K\,e\,y\,w\,o\,r\,d\,s  \ 
   Generalized Poisson integrals, interpolation Lagrange polynomials, best approximations, Lebesgue inequalities, Fourier sums.

\hfill}
 
\newpage

\section{Вступ}\label{intro}

Нехай $L_{p}$,
$1\leq p<\infty$, простір $2\pi$--періодичних сумовних у  $p$-му степені на
 $[0,2\pi)$ функцій $f$, з нормою
\begin{equation*}
\|f\|_{L_p}=\|f\|_{p}:=\Big(\int\limits_{0}^{2\pi}|f(t)|^{p}dt\Big)^{\frac{1}{p}};
\end{equation*}
 $L_{\infty}$ --- простір вимірних і суттєво обмежених   $2\pi$--періодичних функцій $f$ з нормою
 \begin{equation*}
\|f\|_{\infty}=\mathop{\rm{ess}\sup}\limits_{t}|f(t)|;
\end{equation*}
$C$ --- простір неперервних $2\pi$--періодичних функцій  $f$ з нормою
\begin{equation*}
{\|f\|_{C}=\max\limits_{t}|f(t)|}.
\end{equation*}

Через $C^{\alpha,r}_{\beta}\mathfrak{N}, \ \alpha>0, \ r>0, \ \mathfrak{N}\in L_{1}, $ позначимо множину  $2\pi$--періодичних функцій $f(x)$,  які при всіх  
$x\in\mathbb{R}$ можна представити у вигляді згортки  
\begin{equation}\label{conv}
f(x)=\frac{a_{0}}{2}+\frac{1}{\pi}\int\limits_{-\pi}^{\pi}P_{\alpha,r,\beta}(x-t)\varphi(t)dt,
\ a_{0}\in\mathbb{R}, \ \varphi\in\mathfrak{N}, \ \varphi\perp1, 
\end{equation}
з фіксованими ядрами вигляду
\begin{equation}\label{kernel}
P_{\alpha,r,\beta}(t)=\sum\limits_{k=1}^{\infty}e^{-\alpha k^{r}}\cos
\big(kt-\frac{\beta\pi}{2}\big), \ \  \ \alpha,r>0, \ \  \beta\in
    \mathbb{R}.
\end{equation}
Функцію $f$ у рівності \eqref{conv} називають узагальненим інтегралом Пуассона функції $\varphi$ і позначають через $\mathcal{J}^{\alpha,r}_{\beta}\varphi$, з іншого боку функцію $\varphi$ у рівності 
 \eqref{conv} називають узагальненою похідною функції $f$ i позначають через $f^{\alpha,r}_{\beta}$ (тобто, $\varphi(\cdot)=f^{\alpha,r}_{\beta}(\cdot))$. Ядра $P_{\alpha,r,\beta}(\cdot)$ вигляду \eqref{kernel} називають узагальненими ядрами Пуассона.
 
 Зрозуміло, що якщо для заданої функції $\varphi$ виконується рівність \eqref{conv}, то ця ж рівність виконуватиметься і для довільної іншої функції з $L_{1}$, яка може відрізнятись від $\varphi(\cdot)$ на множині міри нуль. Тому надалі рівність $\varphi=f^{\alpha,r}_{\beta}$ домовимось розуміти в тому сенсі, що серед усіх похідних  $f^{\alpha,r}_{\beta}$  є конкретна функція $\varphi$.

 При довільних $r>0$ множини $C^{\alpha,r}_{\beta}\mathfrak{N}$ належать до множини $D^{\infty}$ нескінченно диференційовних $2\pi$--періодичних функцій, тобто $C^{\alpha,r}_{\beta}\mathfrak{N}\subset D^{\infty}$ (див., наприклад, \cite[c. 139]{Stepanets1},  \cite[c. 1408]{Stepanets_Serdyuk_Shydlich2007}). При $r=1$ множини $C^{\alpha,r}_{\beta}\mathfrak{N}$ є множинами звичайних інтегралів Пуассона і складаються із функцій, що допускають регулярне продовження  у смугу $|\mathrm{Im} \, z|< \alpha$ комплексної площини (див., наприклад, \cite[c. 142]{Stepanets1}). При $r>1$  класи $C^{\alpha,r}_{\beta}\mathfrak{N}$ складаються з функцій регулярних в усій комплексній площині (див., наприклад,   \cite[c. 142]{Stepanets1}). Крім того, як випливає з Теореми 1 роботи  \cite{Stepanets_Serdyuk_Shydlich2009} при кожному $r>0$ має місце вкладення $C^{\alpha,r}_{\beta}\mathfrak{N} \subset  J_{1/r} $, де $J_{a}$, $a>0$, --- відомі класи Жевре
 \begin{equation*}
J_{a}=\left\{f\in D^{\infty}: \sup\limits_{k\in\mathbb{N}}\left( \frac{\|f^{(k)} \|_{C}}{(k!)^{a}}\right)^{1/k}<\infty \right\}.
\end{equation*}

Нами вивчаються апроксимативні властивості множин узагальнених інтегралів Пуассона $C^{\alpha,r}_{\beta}\mathfrak{N} $ коли у ролі $\mathfrak{N} $ виступають або усі простори $C$ чи $L_{p}$, $1\leq p\leq \infty$, або одиничні кулі просторів $L_p$, тобто множини $U_{p}=\left\{\varphi\in L_{p}: \ ||\varphi||_{p}\leq 1\right\}$ (далі для зручності класи  $C^{\alpha,r}_{\beta}U_{p}$ будемо позначати через $C^{\alpha,r}_{\beta,p}$), а в якості агрегатів наближення --- класичні інтерполяційні тригонометричні поліноми Лагранжа, що задані непарним числом рівномірно розподілених вузлів.

Для будь--якої $f(x)$ із $ C$ через $\widetilde{S}_{n-1}(f;x)$ будемо позначати тригонометричний поліном порядку $n-1$, що інтерполює $f(x)$ у вузлах $x_{k}^{(n-1)}=\frac{2k\pi}{2n-1}$, $k\in\mathbb{Z}$, тобто такий, що
\begin{equation}\label{InterpolationPolynomS}
\tilde{S}_{n-1}(f;x_{k}^{(n-1)})=f(x_{k}^{(n-1)}), \ k = 0,1, ..., 2n-2.
\end{equation}

Поліноми  $\widetilde{S}_{n-1}(f;\cdot)$ однозначно задаються інтерполяційними умовами \eqref{InterpolationPolynomS}, називаються інтерполяційними поліномами Лагранжа і можуть бути зображені в явному вигляді через ядра Діріхле
\begin{equation*}
D_{n-1}(t)=\frac{1}{2}+\sum\limits_{k=1}^{n-1}\cos kt = \frac{\sin (n-\frac{1}{2})t }{2\sin \frac{t}{2}}
\end{equation*}
наступним чином:
\begin{equation}\label{InterpolationPolynomS_Dirichlet}
\tilde{S}_{n-1}(f;x)= \frac{2}{2n-1}\sum\limits_{k=0}^{2n-2}f(x_{k}^{(n-1)}) D_{n-1}(x-x_{k}^{(n-1)}).
\end{equation}

Нехай $\mathcal{T}_{2n-1}$ --- простір усіх тригонометричних поліномів $t_{n-1}$ порядку $n-1$ і $E_{n}(f)_{L_{p}}$ --- найкраще наближення функції $f\in L_{p}$, $1\leq p\leq \infty$, в $L_{p}$--метриці тригонометричними поліномами $t_{n-1}\in \mathcal{T}_{2n-1}$, тобто величина
\begin{equation*}
E_{n}(f)_{L_{p}}=\inf\limits_{t_{n-1}\in \mathcal{T}_{2n-1}}\|f-t_{n-1}\|_{p}, 
\end{equation*}
а $E_{n}(f)_{C}$ --- найкраще рівномірне наближення функціі $f\in C$ тригонометричними поліномами $t_{n-1}$, тобто величина
\begin{equation*}
E_{n}(f)_{C}=\inf\limits_{t_{n-1}\in \mathcal{T}_{2n-1}}\|f-t_{n-1}\|_{C}.
\end{equation*}

Позначимо через $\tilde{\rho}_{n}(f;\cdot)$ відхилення від функції $f\in C$ її інтерполяційного полінома Лагранжа $\tilde{S}_{n-1}(f;\cdot)$ 
\begin{equation}\label{rhoDef}
 \tilde{\rho}_{n}(f;x)=f(x)- \tilde{S}_{n-1}(f;x).
\end{equation}

Для модулів величин вигляду \eqref{rhoDef} має місце нерівність (див., наприклад, \cite{Korn}, \cite{Stepanets2})
\begin{equation}\label{LebesgueIneq}
\left| f(x)- \tilde{S}_{n-1}(f;x)\right| \leq
(1+ \bar{L}_{n}(x))E_{n}(f)_{C}, \ \ f\in C, \ \ x\in\mathbb{R},
\end{equation}
де
\begin{equation}\label{bar_Ln}
 \bar{L}_{n}(x)
 = \frac{2}{2n-1} \sum\limits_{k=0}^{2n-2}\left|D_{n-1}(x-x_{k}^{(n-1)}) \right|.
\end{equation}

Нерівність \eqref{LebesgueIneq} є інтерполяційним аналогом класичної нерівності Лебега, а функцію $\bar{L}_{n}(x)$ вигляду \eqref{bar_Ln} називають функцією Лебега оператора $\tilde{S}_{n-1}$ вигляду \eqref{InterpolationPolynomS_Dirichlet}.

Асимптотичну поведінку функції Лебега $\bar{L}_{n}(x)$ при $n\rightarrow \infty$ описує наступна формула:
\begin{equation}\label{bar_Ln_Asymp}
 \bar{L}_{n}(x)
 = \frac{2}{\pi} \left| \sin \frac{2n-1}{2}x\right| \ln n + \mathcal{O}(1), \ \ x\in\mathbb{R},
\end{equation}
в якій $ \mathcal{O}(1)$ --- величина, що рівномірно обмежена по $x$ і по $n$. Детальніше  про поведінку констант та функції Лебега оператора  \eqref{kernel} можна ознайомитись у роботах \cite{shakirov2011}, \cite{shakirov2018}.

З урахуванням \eqref{bar_Ln_Asymp} нерівність \eqref{LebesgueIneq} можна записати у вигляді

\begin{equation}\label{LebesgueIneq2}
\left| \tilde{\rho}_{n}(f;x) \right| \leq
\left(\frac{2}{\pi} \left| \sin \frac{2n-1}{2}x\right| \ln n + \mathcal{O}(1) \right)
E_{n}(f)_{C},  \ \ f\in C, \ \ x\in\mathbb{R}.
\end{equation}

Незважаючи на загальність, ця оцінка є асимтотично точною для кожного фіксованого $x\neq \frac{2k\pi}{2n-1}$, $k\in \mathbb{Z}$, на відомих класах $W^{r}$, $r \in \mathbb{N}$, $2\pi$-періодичних функцій, що мають абсолютно неперервні похідні $f^{(k)}$ до $(r-1)$-го порядку включно  і таких, що  $\|f^{(r)}\|_{\infty}\leq 1$. Цей факт випливає із роботи C.М. Нікольського \cite{Nikolsky1945}, в якій на основі  \eqref{LebesgueIneq2} при $r \in \mathbb{N}$ встановлено асимптотичну формулу
\begin{equation}\label{Nikolsky1945}
\widetilde{\mathcal{E}}_{n}(W^{r}_{\infty};x)=
\sup\limits_{f\in W^{r}_{\infty}} \left|  f(x)- \tilde{S}_{n-1}(f;x) \right|
=\frac{2K_{r}}{\pi}\frac{\ln n}{n^{r}} \left| \sin \frac{2n-1}{2}x \right|
+\mathcal{O}\left( \frac{1}{n^{r}}\right),
\end{equation}
де $K_{r}=\frac{4}{\pi}\sum\limits_{v=0}^{\infty}\frac{(-1)^{v(r+1)}}{(2v+1)^{r+1}}$ --- константи Фавара, а величина $\mathcal{O}$ рівномірно обмежена по $x$ і по $n$.

Однак при подальшому збільшенні гладкості і, зокрема, для класів нескінченно диференційовних, аналітичних чи цілих функцій, оцінки відхилень $\left| \tilde{\rho}_{n}(f;x) \right|$, що базуються на використанні \eqref{LebesgueIneq} (чи \eqref{LebesgueIneq2}), перестають бути асимптотично точними і навіть можуть бути не точними за порядком.

Точні порядкові оцінки  $\left \| \tilde{\rho}_{n}(f;x) \right \|_{C}$ на класах 
$$C(\varepsilon)= \left\{ f\in C: \ E_{k}(f)_{C}\leq \varepsilon_{k}, \ \ k\in\mathbb{N} \right\}
$$
 та 
\begin{equation*}
L_{p}(\varepsilon)= \left\{ f\in L_{p}: \ E_{k}(f)_{L_{p}}\leq \varepsilon_{k}, \ \ k\in\mathbb{N} \right\}, \ 1<p<\infty, \ \ \sum\limits_{k=1}^{\infty}\frac{\varepsilon_{k+1}}{k^{1-\frac{1}{p}}}<\infty, \
\end{equation*}
які задаються монотонно прямуючими до нуля послідовностями $\varepsilon=\left\{\varepsilon_{k} \right\}_{k=1}^{\infty}$ невід'ємних чисел, були знайдені у роботах \cite{Oskolkov} та \cite{Sharapudinov}.

У даній роботі для функцій з множин узагальнених інтегралів Пуассона  $C^{\alpha,r}_{\beta}L_{p}$, $\alpha>0$, $r\in(0,1)$, $\beta \in \mathbb{R}$, $1\leq p\leq \infty$, встановлено інтерполяційні аналоги нерівностей типу Лебега, в яких оцінки зверху величин $|\tilde{\rho}_{n}(f;x)|$
виражаються через найкращі наближення $E_{n}(f^{\alpha,r}_{\beta})_{L_{p}}$. 
Також в ній доведено асимптотичну непокращуваність отриманих нерівностей на множинах $C^{\alpha,r}_{\beta}L_{p}$. Слід зауважити, що при $p=\infty$ такі нерівності були встановлені в роботі \cite[Теорема 3]{Serdyuk2004}.

 Крім того, в даній роботі при всіх $x\in\mathbb{R}$, $\alpha>0$, $\beta\in\mathbb{R}$, $r\in(0,1)$,  $1\leq p\leq \infty$, розв'язано задачу Колмогорова-Нікольського для інтерполяційних поліномів Лагранжа $ \tilde{S}_{n-1}(f;x) $ вигляду \eqref{InterpolationPolynomS_Dirichlet}  на класах узагальнених інтегралів Пуассона, тобто встановлено асмиптотичні при
 $n\rightarrow\infty$ рівності для величин
\begin{equation}\label{quantityInterpol}
\tilde{\mathcal{E}}_{n}(C^{\alpha,r}_{\beta,p};x)=\sup\limits_{f\in C^{\alpha,r}_{\beta,p} } \left| \tilde{\rho}_{n}(f;x) \right|.
\end{equation}

Зазначимо, що при $r \geq 1$, $1\leq p\leq \infty$, асимптотичні рівності для зазначених величин були знайдені в роботах \cite{SerdyukDopov1999}, \cite{StepanetsSerdyuk2000},  \cite{Serdyuk2012}, \cite{SerdyukVoitovych2010}.

У роботі  \cite{StepanetsSerdyuk2000} було показано, що якщо $r=1$, $p=\infty$, $\alpha>0$, $\beta\in \mathbb{R}$,  $x\in\mathbb{R}$, то при $n\rightarrow\infty$ має місце асимптотична рівність
\begin{align}\label{StepanetsSerdyukr=1p=infty}
\tilde{\mathcal{E}}_{n}(C^{\alpha,1}_{\beta,\infty};x)
=
e^{-\alpha n^{r}} \left|\sin \frac{2n-1}{2}x \right| 
\left(
\frac{16}{\pi^{2}} \mathbf{K}(e^{-\alpha})
+ \mathcal{O}(1)\frac{e^{-\alpha n}}{n(1-e^{-\alpha n})}
\right),
\end{align}
в якій $ \mathbf{K}(q)=\int\limits_{0}^{\frac{\pi}{2}}\frac{du}{\sqrt{1-q^{2}\sin^{2}u}}$ --- повний еліптичний інтеграл першого роду, а $\mathcal{O}(1)$ --- величина рівномірно обмежена по  $n$, $x$, $\alpha$ i $\beta$.

Як випливає з \cite{Serdyuk2012},  для величин вигляду \eqref{quantityInterpol} при всіх $\alpha>0$ i $\beta \in \mathbb{R}$ у випадку $r=1$ i $1<p\leq\infty$  виконується асимптотична при $n\rightarrow\infty$ рівність
\begin{align}\label{quantityInterpolr=1}
&\tilde{\mathcal{E}}_{n}(C^{\alpha,1}_{\beta,p};x) \notag \\
=&
e^{-\alpha n} \! \left|\sin \frac{2n\!-\!1}{2}x \right| \! \! \left(\frac{2}{\pi}\| \cos t\|_{p'}F^{1/p'}\left(\frac{p'}{2},  \frac{p'}{2}; 1; e^{-2\alpha} \right) \! + \!
\mathcal{O}(1)\frac{e^{-\alpha}}{n(1\!-\!e^{-\alpha})^{s(p)}}  \right),  x\in\mathbb{R}
\end{align}
в якій $\frac{1}{p}+\frac{1}{p'}=1$, $F(a,b;c;z)$ --- гіпергеометрична функція Гаусса
\begin{equation*}
F(a,b;c;z)=1+\sum\limits_{k=1}^{\infty}\frac{(a)_{k}(b)_{k}}{(c)_{k}}\frac{z^{k}}{k!},
\end{equation*}
\begin{equation*}
(y)_{k}:=y(y+1)(y+2)...(y+k-1),
\end{equation*}
 $s(p)$ задається формулою
\begin{equation*}
s(p)={\left\{\begin{array}{cc}
 1,  & p=\infty, \\
2, & 1\leq  p<\infty,
  \end{array} \right.}
\end{equation*}
а у випадку $r=1$ i $p=1$ --- рівність
\begin{equation}\label{quantityInterpolr=1p=1}
\tilde{\mathcal{E}}_{n}(C^{\alpha,1}_{\beta,1};x)=
e^{-\alpha n}\left|\sin \frac{2n-1}{2}x \right|\left(\frac{2}{\pi} \frac{1}{1-e^{-\alpha}} +
\mathcal{O}(1)\frac{e^{-\alpha}}{n(1-e^{-\alpha})^{2}}  \right), \ \ x\in\mathbb{R}.
\end{equation}

У формулах \eqref{quantityInterpolr=1} i \eqref{quantityInterpolr=1p=1} величини $\mathcal{O}(1)$ рівномірно обмежені відносно параметрів $x$, $n$, $\beta$, $\alpha$ i $p$.

Оскільки при $p=\infty$ ($p'=1$) $\|\cos t\|_{p'}=\|\cos t \|_{1}=4$
і
\begin{equation*}
F^{\frac{1}{p'}}\left(\frac{p'}{2},\frac{p'}{2};1;  e^{-2\alpha} \right)=
F  \left(\frac{1}{2},\frac{1}{2};1; e^{-2\alpha} \right)=\frac{2}{\pi}\mathbf{K}(e^{-\alpha}),
\end{equation*}
то з \eqref{quantityInterpolr=1} випливає \eqref{StepanetsSerdyukr=1p=infty}.

Зауважимо також, що  в роботі \cite{SerdyukSokolenko2016} для величини виду \eqref{quantityInterpol} при $r=1$, $p=2$,  $\alpha>0$, $\beta\in\mathbb{R}$ i $n\in\mathbb{N}$ встановлено рівність
\begin{align}\label{quantityInterpolr=1p=2}
&\tilde{\mathcal{E}}_{n}(C^{\alpha,r}_{\beta,2};x)
\notag \\
=&e^{-\alpha n}\left|\sin \frac{2n-1}{2}x \right| \frac{2}{\sqrt{\pi(1-e^{-2\alpha})}}
\left(
\frac{1+e^{-2\alpha(2n-1)}}{1-2e^{-2\alpha(2n-1)}\cos(2n-1)x+e^{-4\alpha(2n-1)}}
\right)^{\frac{1}{2}},  \ x\in\mathbb{R}.
\end{align}

Більше того, як випливає з \cite{SerdyukSokolenko2016} i  \cite{SerdyukSokolenko2017},  при $p=2$ та всіх $r>0$, $\alpha>0$, $\beta\in\mathbb{R}$ i $n\in\mathbb{N}$ для величин $\tilde{\mathcal{E}}_{n}(C^{\alpha,r}_{\beta,2};x)$ має місце рівність
\begin{equation}\label{quantityInterpolp=2}
\tilde{\mathcal{E}}_{n}(C^{\alpha,r}_{\beta,2};x)=
\frac{2}{\sqrt{\pi}}\left(\sum\limits_{m=1}^{\infty}\sin^{2}\frac{(2n-1)mx}{2}\sum\limits_{k=m(2n-1)-n+1}^{m(2n-1)+n-1}e^{-2\alpha k^{r}} \right)^{\frac{1}{2}}
, \ \ x\in\mathbb{R}
\end{equation}
що є справедливою при всіх $\beta\in\mathbb{R}$ i $n\in\mathbb{N}$.

У випадку $r>1$, як випливає з \cite{Serdyuk2012}, \cite{SerdyukVoitovych2010}, для величин $\tilde{\mathcal{E}}_{n}(C^{\alpha,r}_{\beta,p};x)$, $\alpha>0$, $\beta\in\mathbb{R}$ при $p=\infty$ має місце асимптотична при $n\rightarrow\infty$ рівність
\begin{align}\label{quantityInterpolp=infty}
&\tilde{\mathcal{E}}_{n}(C^{\alpha,r}_{\beta,\infty};x)
\notag \\
=&e^{-\alpha n^{r}}\left|\sin \frac{2n-1}{2}x \right| 
\left(
\frac{8}{\pi}+
\mathcal{O}(1)\left(\frac{e^{2\alpha n^{r}}}{e^{2\alpha(n+1)^{r}}}+
\left(1+\frac{1}{\alpha r(n+2)^{r-1}} \right)\frac{e^{\alpha n^{r}}}{e^{\alpha(n+1)^{r}}} \right)
\right),  x\in\mathbb{R}.
\end{align}
а при $1\leq p<\infty$ --- рівність
\begin{align}\label{quantityInterpolp<infty}
&\tilde{\mathcal{E}}_{n}(C^{\alpha,r}_{\beta,p};x)
\notag \\
=&e^{-\alpha n^{r}}\left|\sin \frac{2n-1}{2}x \right| 
\left(
\frac{2}{\pi}\|\cos t \|_{p'}+
\mathcal{O}(1)\left(1+\frac{1}{\alpha r(n+1)^{r-1}} \right)\frac{e^{\alpha n^{r}}}{e^{\alpha(n+1)^{r}}} 
\right),  x\in\mathbb{R}.
\end{align}

У формулах \eqref{quantityInterpolp=infty} i \eqref{quantityInterpolp<infty} величини $\mathcal{O}(1)$ рівномірно обмежені по $x$, $n$, $r$, $\alpha$, $\beta$ i $p$.

Зазначимо також, що в роботі \cite{SerdyukSokolenko2019} для класів $C^{\alpha,r}_{\beta,1}$, $\alpha>0$, $r>1$, $\beta\in\mathbb{R}$ встановлено і асимптотичні рівності для точних верхніх меж відхилень інтерполяційних поліномів $\tilde{S}_{n-1}(f;\cdot)$ в довільних $L_{p}$--метриках ($1\leq p \leq \infty$).

Що ж стосується випадку $0<r<1$, то асимптотичні рівності для величин 
$\tilde{\mathcal{E}}_{n}(C^{\alpha,r}_{\beta,p};x)$, $\alpha>0$, $\beta\in\mathbb{R}$, за виключенням наведеного вище випадку $p=2$, були відомі лише у випадку  $p=\infty$ завдяки роботам \cite{StepanetsSerdyuk2000Zb} та \cite{Serdyuk2004},  з яких випливає, що при $n\rightarrow\infty$
\begin{equation}\label{quantityInterpolp=infty1}
\tilde{\mathcal{E}}_{n}(C^{\alpha,r}_{\beta,\infty};x)
=e^{-\alpha n}\left|\sin \frac{2n-1}{2}x \right| 
\left(
\frac{8}{\pi^{2}}\ln n^{1-r}+
\mathcal{O}(1)
\right),  \ \ x\in\mathbb{R},
\end{equation}
де $\mathcal{O}(1)$ --- величина рівномірно обмежена по $x$, $n$ i $\beta$.

В даній роботі буде доведено зокрема, що для довільних $0<r<1$, $\alpha>0$, $\beta\in\mathbb{R}$,  i $x\in\mathbb{R}$ при $1<p<\infty$ та  $n\rightarrow\infty$ має місце асимптотична рівність

\begin{align}\label{quantityInterpol1<p<infty}
&\tilde{\mathcal{E}}_{n}(C^{\alpha,r}_{\beta,p};x)
\notag \\
=&e^{-\alpha n^{r}}n^{\frac{1-r}{p}}\left|\sin \frac{2n-1}{2}x \right| 
\left(
\frac{2\|\cos t \|_{p'}}{\pi^{1+\frac{1}{p'}}(\alpha r)^{\frac{1}{p}}}F^{\frac{1}{p'}}\left(\frac{1}{2},\frac{3-p'}{2};\frac{3}{2};1 \right)+
\mathcal{O}(1) \frac{1}{n^{\min \{r, \frac{1-r}{p} \}}}
\right),  
\end{align}
де $\frac{1}{p}+\frac{1}{p'}=1$,  $F(a,b;c;z)$ --- гіпергеометрична функція Гаусса, $\mathcal{O}(1)$ --- величина рівномірно обмежена по $x$, $n$ i $\beta$, а при $p=1$ --- рівність
\begin{align}\label{quantityInterpolp=1_1}
\tilde{\mathcal{E}}_{n}(C^{\alpha,r}_{\beta,1};x) 
= e^{-\alpha n^{r}}n^{1-r}
\left|\sin \frac{2n-1}{2}x \right| 
\Big(
\frac{2}{\pi\alpha r}+\mathcal{O}(1) \frac{1}{n^{\min \{r, 1-r \}}}\Big).
\end{align}
 При цьому у роботі в явному вигляді записано оцінки залишкового члена у формулах \eqref{quantityInterpol1<p<infty} i \eqref{quantityInterpolp=1_1} через параметри задачі, що може бути корисним для практичного застосування отриманих в ній результатів. Отже, на класах узагальнених інтегралів  Пуассона $C^{\alpha,r}_{\beta,p}$ при всіх $\alpha>0$, $r>0$, $\beta\in \mathbb{R}$ i $1\leq p \leq\infty$ повністю розв'язано задачу Колмогорова-Нікольського для інтерполяційних поліномів Лагранжа, яка полягає  у встановленні для кожного $x\in\mathbb{R}$ сильної асимптотики величин $\tilde{\mathcal{E}}_{n}(C^{\alpha,r}_{\beta,p};x)$ вигляду \eqref{quantityInterpol} при $n\rightarrow\infty$.

Головний член $A_{n}$ в асимптотичному розкладі величини \eqref{quantityInterpol}, поданому у вигляді
\begin{equation*}
 \widetilde{{\cal E}}_{n}(C^{\alpha,r}_{\beta,p};x)=e^{-\alpha n^{r}}\Big|\sin\frac{2n-1}{2}x\Big|(A_n+o(A_n)).
\end{equation*}
природно назвати константами Колмогорова-Нікольського для інтерполяційних поліномів Лагранжа на класах $C^{\alpha,r}_{\beta,p}$. Наступна таблиця містить точні значення зазначених констант  в залежності від співвідношень  між параметрами  $r$ i $p$:

\begin{center}\label{Tabl}

\begin{small}
\begin{tabular}{|c|c|c|c|c|}
	\hline
	\multicolumn{2}{|c|}{}  & \multicolumn{3}{|c|}{\boldmath{$r$}}  \\
\cline{3-5}
\multicolumn{2}{|c|}{\raisebox{1.5ex}[0cm][0cm]{$A_n$}}   & \boldmath{$(0,1)$} & \boldmath{$1$} & \boldmath{$(1,\infty)$} \\
\hline
 \ &  & Степанець, Сердюк (2000)   \cite{StepanetsSerdyuk2000Zb}    &  Степанець, Сердюк (2000)  \cite{StepanetsSerdyuk2000}  & Сердюк (1999)  \cite{SerdyukDopov1999}    \\
  \ & \boldmath{$\infty$} &  Сердюк  (2004)    \cite{Serdyuk2004}   &    &  Степанець,  \\
  \ &  &       &    &  Сердюк   (2000) ] \cite{StepanetsSerdyuk2000}    \\
       &   & \raisebox{-0.9ex}[0cm][0cm]{ $\frac{8}{\pi^{2}}(1-r)\ln n$   } & \raisebox{-0.9ex}[0cm][0cm]{$\frac{16}{\pi^{2}} {\bf K} (e^{-\alpha})$} &  \raisebox{-0.9ex}[0cm][0cm]{$\frac{8}{\pi}$ }\\
     \tiny{ \ }  &  \tiny{ \ }  &   \tiny{ \ } & \tiny{ \ } &  \tiny{ \ }  \\
\cline{2-5}
 \raisebox{-1.4ex}[0cm][0cm]{ \boldmath{$p$} }& \ 
 &    \textbf{Результати авторів роботи}     &   Сердюк (2012) \cite{Serdyuk2012}  & Сердюк, \\
  & \boldmath{$(1, \ \infty)$}  &     & 
  & Войтович (2010) \cite{SerdyukVoitovych2010}  \\
                       &  &   \boldmath{$ n^{\frac{1-r}{p}} 
\frac{2\|\cos t \|_{p'}}{\pi^{1+\frac{1}{p'}}(\alpha r)^{\frac{1}{p}}}F^{\frac{1}{p'}}\left(\frac{1}{2},\frac{3-p'}{2};\frac{3}{2};1 \right)$}   &    $\frac{2\|\cos t\|_{p'}}{\pi} F^{\frac{1}{p'}}(\frac{p'}{2}, \frac{p'}{2}; 1; e^{-2\alpha})$                        & $\frac{2\|\cos t\|_{p'}}{\pi}$ \\
  &  &    &                             &                      \\
 \cline{2-5}
   & 
   &    \textbf{Результати авторів роботи}  &  Сердюк (2012) \cite{Serdyuk2012}  &  Сердюк,   \\
   & \boldmath{$1$} &   \raisebox{-0.9ex}[0cm][0cm]{ } & \raisebox{-0.9ex}[0cm][0cm] 
   & \raisebox{-0.9ex}[0cm][0cm]{Войтович (2010)  \cite{SerdyukVoitovych2010} }\\
                        &  & \boldmath{ $ n^{1-r}
\frac{2}{\pi\alpha r}$} & {$\frac{2}{\pi(1-e^{-\alpha})} $}  & \raisebox{-0.9ex}[0cm][0cm]{$\frac{2}{\pi}$}\\
  \tiny{ \ }  &  \tiny{ \ }  &   \tiny{ \ } & \tiny{ \ } &  \tiny{ \ } \\
 \hline
\end{tabular}
\end{small}
\end{center}

\section{Нерівності типу Лебега для інтерполяційних поліномів Лагранжа на множинах узагальнених інтегралів Пуассона}

При довільних фіксованих $\alpha>0$, $r\in(0,1)$ i $1\leq p\leq\infty$ позначимо через $n_{*}=n_{*}(\alpha,r,p)$ найменший з номерів $n$ такий, що

\begin{equation}\label{n_star}
 \frac{\ln \pi n}{\alpha rn^{r}}+\frac{\alpha r \chi(p)}{n^{1-r}}\leq{\left\{\begin{array}{cc}
 \frac{1}{14},  & p=1, \\
\frac{1}{(3\pi)^3}\cdot\frac{p-1}{p}, & 1< p<\infty, \\
\frac{1}{(3\pi)^3},  & p=\infty,
  \end{array} \right.}
\end{equation}
де $\chi(p)=p$ при $1\leq p<\infty$ i $\chi(p)=1$ при $p=\infty$.

\begin{theorem}\label{theorem_1<p<infty}
Нехай $0<r<1$,  $\alpha>0$, $1< p<\infty$, $\beta\in\mathbb{R}$ і $n\in \mathbb{N}$. Тоді, для всіх $x\in\mathbb{R}$ і довільної  функції
 $f\in C^{\alpha,r}_{\beta}L_{p}$ при $n\geq n_{*}(\alpha,r,p)$  має місце нерівність
\begin{align}\label{Theorem_Case1<p<infty}
&|\tilde{\rho}_{n}(f;x)|
\leq
2e^{-\alpha n^{r}}n^{\frac{1-r}{p}}\left|\sin \frac{2n-1}{2}x \right| 
\left(
\frac{\|\cos t \|_{p'}}{\pi^{1+\frac{1}{p'}}(\alpha r)^{\frac{1}{p}}}F^{\frac{1}{p'}}\left(\frac{1}{2},\frac{3-p'}{2};\frac{3}{2};1 \right) \right.
\notag \\
+&\left.
\gamma^{*}_{n,p}\left( 
\left(1+\frac{(\alpha r)^{\frac{p'-1}{p}}}{p'-1}\right)\frac{1}{n^{\frac{1-r}{p}}}
+\frac{p^{\frac{1}{p'}}}{(\alpha r)^{1+\frac{1}{p}}n^{r}}
\right)
\right) E_{n}(f^{\alpha,r}_{\beta})_{L_{p}}.
\end{align}

Крім того, для довільної функції $f\in C^{\alpha,r}_{\beta}L_{p}$ можна вказати функцію $\mathcal{F}(\cdot)=\mathcal{F}(f;n;x, \cdot)$, таку, що $E_{n}(\mathcal{F}^{\alpha,r}_{\beta})_{L_p}=E_{n}(f^{\alpha,r}_{\beta})_{L_p}$  і для $n\geq n_{*}(\alpha, r,p)$ виконується наступна рівність:
\begin{align}\label{Theorem_Case1<p<inftyEquality}
&|\tilde{\rho}_{n}(\mathcal{F}; x)|
=
2e^{-\alpha n^{r}}n^{\frac{1-r}{p}}\left|\sin \frac{2n-1}{2}x \right| 
\left(
\frac{\|\cos t \|_{p'}}{\pi^{1+\frac{1}{p'}}(\alpha r)^{\frac{1}{p}}}F^{\frac{1}{p'}}\left(\frac{1}{2},\frac{3-p'}{2};\frac{3}{2};1 \right) \right.
\notag \\
+&\left.
\gamma^{*}_{n,p}\left( 
\left(1+\frac{(\alpha r)^{\frac{p'-1}{p}}}{p'-1}\right)\frac{1}{n^{\frac{1-r}{p}}}
+\frac{p^{\frac{1}{p'}}}{(\alpha r)^{1+\frac{1}{p}}n^{r}}
\right)
\right) E_{n}(f^{\alpha,r}_{\beta})_{L_{p}} .
\end{align}

В \eqref{Theorem_Case1<p<infty}  і \eqref{Theorem_Case1<p<inftyEquality}  величини ${\gamma_{n,p}^{*}=\gamma_{n,p}^{*}(\alpha,r,\beta,f,x)}$ такі, що ${|\gamma_{n,p}^{*}|<20\pi^{4}}$.
\end{theorem}

\begin{proof}[Доведення Теореми~\ref{theorem_1<p<infty}]

Згідно з Лемою 1 роботи \cite{StepanetsSerdyuk2000} для довільної функції $f\in C^{\alpha,r}_{\beta}L_{p}$, $1\leq p\leq\infty$, $\alpha>0$, $r>0$, $\beta\in\mathbb{R}$ у кожній точці $x\in\mathbb{R}$ має місце наступне інтегральне зображення величини $\tilde{\rho}_{n}(f;x)$:
\begin{equation}\label{IntegrRepr}
\tilde{\rho}_{n}(f;x)=\frac{2}{\pi}\sin\frac{2n-1}{2}x 
\int\limits_{-\pi}^{\pi}\delta_{n}(t+x)
\left(\sum\limits_{k=n}^{\infty}e^{-\alpha k^{r}}\cos (kt+\gamma_{n})+r_{n}(t)
 \right) dt,
\end{equation}
в якому $\delta_{n}(\tau)= f^{\alpha,r}_{\beta}(\tau)-t_{n-1}(\tau)$, $t_{n-1}$ --- довільний тригонометричний поліном із множини $\mathcal{T}_{2n-1}$, а $r_{n}$ i $\gamma_{n}$ означені за допомогою рівностей 
\begin{equation}\label{rn}
r_{n}(t)=r_{n}(\alpha;r;\beta;x;t)=
\sum\limits_{k=1}^{\infty}\sum\limits_{\nu=(2k+1)n-k}^{\infty}\!\!\! e^{-\alpha \nu^{r}}\sin\left(\nu t+ \left(k+\frac{1}{2} \right)(2n-1)x+ \frac{\beta\pi}{2}\right),
\end{equation}
\begin{equation}\label{gamma_n}
\gamma_{n}=\gamma_{n}(\beta;x)=
\frac{(2n-1)x+\pi(\beta-1)}{2}.
\end{equation}

Для знаходження оцінки зверху абсолютної величини залишкового члена $r_{n}(t)$ у формулі \eqref{IntegrRepr} нам буде корисним наступне твердження.

\begin{lemma}\label{Lemma1}
Нехай $\alpha>0$, $r\in(0,1)$, а номер $n$, $n\in\mathbb{N}$ такий, що виконується нерівність
\begin{equation}\label{NumberIneq}
\frac{1}{\alpha r n^{r}}+\frac{\alpha r}{n^{r-1}}\leq \frac{1}{14}.
\end{equation}
Тоді
\begin{equation}\label{IneqLem1}
\sum\limits_{k=1}^{\infty}\sum\limits_{v=(2k+1)n-k}^{\infty} e^{-\alpha v^{r}}<
\frac{636}{169} \frac{n^{1-r}}{\alpha r}e^{-\alpha (3n-1)^{r}}.
\end{equation}
\end{lemma}
Доведення Леми~\ref{Lemma1} нами  наведено в підрозділі~\ref{LemmaProof} даної роботи.

Співставивши нерівності \eqref{n_star} i \eqref{NumberIneq} легко переконатись, що  якщо $n\geq n_{*}(\alpha,r,p)$ при довільних фіксованих $\alpha>0$, $r\in(0,1)$ i $1\leq p\leq \infty$, то при вказаних $n$, $\alpha$ i $r$ умова \eqref{NumberIneq} Леми~\ref{Lemma1} також виконується, а разом з нею --- і нерівність \eqref{IneqLem1}.

Тому,  з урахуванням \eqref{rn}, при $n\geq n_{*}(\alpha,r,p)$, $1\leq p\leq \infty$, одержуємо
\begin{equation}\label{rnEstimate}
|r_{n}(t)| \leq
\sum\limits_{k=1}^{\infty}\sum\limits_{v=(2k+1)n-k}^{\infty} e^{-\alpha v^{r}} < \frac{636}{169} \frac{n^{1-r}}{\alpha r}e^{-\alpha (3n-1)^{r}}.
\end{equation}

Покажемо, що при довільних $n\geq n_{*}(\alpha,r,p)$, $r\in(0,1)$, $\alpha>0$, $1\leq p\leq \infty$, 
\begin{equation}\label{AdditEstimate1}
\frac{n^{1-r}}{\alpha r}<\frac{1}{\pi} e^{\alpha((3n-1)^{r}-n^{r})}.
\end{equation}
Дійсно, в силу \eqref{n_star}
\begin{equation}\label{AdditEstimate2}
\frac{ \ln (\pi n)}{\alpha r n^{r}}\leq \frac{1}{14}, \ \frac{n^{1-r}}{\alpha r}\geq 14,
\end{equation}
а, отже,
\begin{equation*}
\frac{\ln \frac{\pi n}{\alpha r n^{r}}}{\alpha r n^{r}}<
\frac{\ln \frac{\pi n}{\alpha r n^{r}}+\ln \alpha r n^{r}}{\alpha r n^{r}}
=
\frac{\ln  \pi n}{\alpha r n^{r}}\leq \frac{1}{14},
\end{equation*}
звідки
\begin{equation*}
\ln \frac{\pi n}{\alpha r n^{r}} \leq \frac{\alpha r n^{r}}{14},
\end{equation*}
або, що те, саме 
\begin{equation}\label{AdditEstimate4}
 \frac{\pi n}{\alpha r n^{r}} \leq e^{\frac{\alpha r n^{r}}{14}}.
\end{equation}

Оскільки
\begin{equation*}
\frac{r}{2^{1-r}}< 2^{r}-1<r, \ \ r\in(0,1),
\end{equation*}
то
\begin{equation}\label{AdditEstimate5}
e^{\frac{\alpha r n^{r}}{14}}
<
e^{\frac{\alpha r n^{r}}{2^{1-r}}}<e^{\alpha n^{r}(2^{r-1})}\leq 
e^{\alpha n^{r}((3-\frac{1}{n})^{r}-1)} =  e^{\alpha((3n-1)^{r}-n^{r})}.
\end{equation}

Об'єднавши \eqref{AdditEstimate4} i \eqref{AdditEstimate5}, отримуємо \eqref{AdditEstimate1}.

Із \eqref{rnEstimate} і \eqref{AdditEstimate1} випливає наступна оцінка для $|r_{n}(t)|$:
\begin{equation}\label{rNAdditEstimate}
|r_{n}(t)|< \frac{636}{169\pi}e^{-\alpha n^{r}}, \ \ n\geq n_{*}(\alpha,r,p), \ \ \alpha>0, \ r\in(0,1), \ 1\leq p\leq\infty.
\end{equation}

Беручи в \eqref{IntegrRepr} в якості $t_{n-1}$ поліном $t_{n-1}^{*}$ найкращого наближення у просторі $L_{p}$ функції $f^{\alpha,r}_{\beta}(\cdot)$, тобто такий, що
\begin{equation}\label{bestApprox}
\| f^{\alpha,r}_{\beta} - t_{n-1}^{*} \|_{p} = E_{n}(f^{\alpha,r}_{\beta} )_{L_{p}}
=\inf\limits_{t_{n-1} \in \mathcal{T}_{2n-1}}\| f^{\alpha,r}_{\beta} - t_{n-1} \|_{p} , \ \ 1\leq p\leq\infty,
\end{equation}
і застосовуючи нерівність Гельдера
\begin{equation}\label{HolderIneq}
\int\limits_{-\pi}^{\pi}|h(t)g(t)|dt \leq \|h\|_{p}\|g\|_{p'}, \ \ h\in L_{p}, \ \ 1\leq p\leq \infty,
g\in L_{p'}, \ \ \frac{1}{p}+\frac{1}{p'}=1
\end{equation}

та оцінку \eqref{rNAdditEstimate},  для довільної функції $f\in C^{\alpha,r}_{\beta}L_{p}$ при $ n\geq n_{*}(\alpha,r,p)$ маємо

\begin{equation}\label{AdditionEstimRho}
|\tilde{\rho}_{n}(f;x)| \leq 2 \left| \sin \frac{2n-1}{2}x\right| 
\left(\frac{1}{\pi} \left\| \sum\limits_{k=n}^{\infty}e^{-\alpha k^{r}}\cos (kt+\gamma_{n})   \right\|_{p'} +\theta_{n,p}e^{-\alpha n^{r}}
\right)E_{n}(f^{\alpha,r}_{\beta})_{L_{p}},
\end{equation}
де $\gamma_{n}$ означена формулою \eqref{gamma_n}, а для величини $\theta_{n,p}=\theta_{n,p}(\alpha,r,\beta,x)$ виконується оцінка 
$|\theta_{n,p}|< \frac{1272}{169 \pi}$, $1\leq p\leq\infty$.

Із \cite{SerdyukStepanyukDopov}--\cite{SerdyukStepanyuk2017} випливає, що при довільних $r\in(0,1)$, $\alpha>0$, $\xi\in\mathbb{R}$, $1\leq p \leq \infty$ і $n\geq n_{0}(\alpha, r ,p)$, де $n_{0}(\alpha, r ,p)$ --- найменший з номерів $n$, такий, що

\begin{equation}\label{n_0}
 \frac{1}{\alpha rn^{r}}+\frac{\alpha r \chi(p)}{n^{1-r}}\leq{\left\{\begin{array}{cc}
 \frac{1}{14},  & p=1, \\
\frac{1}{(3\pi)^3}\cdot\frac{p-1}{p}, & 1< p<\infty, \\
\frac{1}{(3\pi)^3},  & p=\infty,
  \end{array} \right.}
\end{equation}
мають місце оцінки
\begin{align}\label{EstimNorm1}
&\frac{1}{\pi}\left \| \sum\limits_{k=n}^{\infty}e^{-\alpha k^{r}}\cos(kt+\xi) \right \|_{p'}
\notag \\
=&e^{-\alpha n^{r}}n^{\frac{1-r}{p}} \left( \frac{\| \cos t\|_{p'} }{\pi^{1+\frac{1}{p'}}(\alpha r)^{\frac{1}{p}}}I_{p'}\left(\frac{\pi n^{1-r}}{\alpha r} \right)+
\gamma_{n,p}^{(1)}\left( \frac{1}{(\alpha r)^{1+\frac{1}{p}}}I_{p'}\left(\frac{\pi n^{1-r}}{\alpha r} \right) \frac{1}{n^{r}}+\frac{1}{n^{\frac{1-r}{p}}}\right) \right),
\end{align}
\begin{align}\label{EstimNorm2}
&\frac{1}{\pi}\inf\limits_{\lambda\in \mathbb{R}}\left \| \sum\limits_{k=n}^{\infty}e^{-\alpha k^{r}}\cos(kt+\xi) -\lambda \right \|_{p'}
\notag \\
=&e^{-\alpha n^{r}}n^{\frac{1-r}{p}} \left( \frac{\| \cos t\|_{p'} }{\pi^{1+\frac{1}{p'}}(\alpha r)^{\frac{1}{p}}}I_{p'}\left(\frac{\pi n^{1-r}}{\alpha r} \right)+
\gamma_{n,p}^{(2)}\left( \frac{1}{(\alpha r)^{1+\frac{1}{p}}}I_{p'}\left(\frac{\pi n^{1-r}}{\alpha r} \right) \frac{1}{n^{r}}+\frac{1}{n^{\frac{1-r}{p}}}\right) \right),
\end{align}
в яких $\frac{1}{p}+\frac{1}{p'}=1$, $I_{s}(v):=\left\| \frac{1}{\sqrt{1+t^{2}}} \right\|_{L_{s}[0,v]}$,
\begin{equation}\label{Is_norm}
I_{s}(v):=\left\| \frac{1}{\sqrt{1+t^{2}}} \right\|_{L_{s}[0,v]}
=
   {\left\{\begin{array}{cc}
\bigg(\int\limits_{0}^{v}  \left| \frac{1}{\sqrt{1+t^{2}}} \right| ^{s}dt
\bigg)^{\frac{1}{s}}, & 1\leq s<\infty, \\
\mathop{\rm{ess}\sup}\limits_{t\in[0,v]} \left| \frac{1}{\sqrt{1+t^{2}}} \right|, \ & s=\infty. \
  \end{array} \right.}
\end{equation}
а для величин $\gamma_{n,p}^{(i)}=\gamma_{n,p}^{(i)}(\alpha, r,\xi)$, $i=1,2$ виконуються нерівності $|\gamma_{n,p}^{(i)}|\leq(14\pi)^{2}$.

Враховуючи, що згідно з \eqref{n_star} i \eqref{n_0} $n_{0}(\alpha,r,p) \leq n_{*}(\alpha,r,p)$, то, застосувавши формулу \eqref{EstimNorm1} при $\xi=\gamma_{n}$, де $\gamma_{n}$ означена формулою \eqref{gamma_n}, із \eqref{HolderIneq} i \eqref{AdditionEstimRho} при $n\geq n_{*}(\alpha,r,p)$ отримуємо
\begin{align}\label{RhoEstimate2}
|\tilde{\rho}_{n} (f;x)|\leq& 2 e^{-\alpha n^{r}}n^{\frac{1-r}{p}} \left| \sin \frac{2n-1}{2}x \right|
\left(
\frac{\|\cos t \|_{p'}}{\pi^{1+\frac{1}{p'}}(\alpha r)^{\frac{1}{p}}} I_{p'}\left(\frac{\pi n^{1-r}}{\alpha r} \right) \right.
\notag \\
+&
\left.
\frac{\gamma_{n,p}^{(1)}}{(\alpha r)^{1+\frac{1}{p}}}I_{p'}\left(\frac{\pi n^{1-r}}{\alpha r} \right) \frac{1}{n^{r}}+
\left(\gamma_{n,p}^{(1)}+\theta_{n,p}  \right)\frac{1}{n^{\frac{1-r}{p}}}
\right) E_{n}(f^{\alpha,r}_{\beta})_{L_{p}}, \ \ 1\leq p \leq \infty.
\end{align}


Як встановлено в \cite{SerdyukStepanyuk2017}, при $1<p<\infty$ і $n\geq n_{0}(\alpha,r,p)$ 
\begin{equation}\label{estimateIp}
I_{p'}\left(\frac{\pi n^{1-r}}{\alpha r} \right)
=
F^{\frac{1}{p'}}\left(\frac{1}{2}, \frac{3-p'}{2}; \frac{3}{2}; 1 \right)+
\frac{\Theta_{\alpha,r,p,n}^{(1)}}{p'-1}\left(\frac{\alpha r}{\pi n^{1-r}} \right)^{p'-1},
\end{equation}
де $|\Theta_{\alpha,r,p,n}^{(1)}|<2$, і крім того, 
\begin{equation}\label{IneqIp}
I_{p'}\left(\frac{\pi n^{1-r}}{\alpha r} \right) <
p^{\frac{1}{p'}}.
\end{equation}
Із \eqref{estimateIp}, \eqref{IneqIp}, а також  з очевидної нерівності
\begin{equation*}
\frac{1}{n^{\frac{1-r}{p}}}>
\frac{1}{n^{(1-r)(p'-1)} }
\end{equation*}
випливає, що при $n\geq n_{0}(\alpha,r,p)$, $1<p<\infty$, $\xi\in \mathbb{R}$, співвідношення \eqref{EstimNorm1} i \eqref{EstimNorm2} приводять до наступних оцінок:
\begin{align}\label{EstimNorm3}
&\frac{1}{\pi}\left \| \sum\limits_{k=n}^{\infty}e^{-\alpha k^{r}}\cos(kt+\xi) \right \|_{p'}
= e^{-\alpha n^{r}}n^{\frac{1-r}{p}}  \left( \frac{\| \cos t\|_{p'} }{\pi^{1+\frac{1}{p'}}(\alpha r)^{\frac{1}{p}}} F^{\frac{1}{p'}}\left(\frac{1}{2}, \frac{3-p'}{2}; \frac{3}{2}; 1 \right)
\right.
\notag \\
+&
\left.
\bar{\gamma}_{n,p}^{(1)}
\left(\left(1+ \frac{(\alpha r)^{\frac{p'-1}{p}}}{p'-1} \right) 
 \frac{1}{n^{\frac{1-r}{p}}}
+
\frac{p^{\frac{1}{p'}}}{(\alpha r)^{1+\frac{1}{p}}n^{r}} 
  \right)\right),
\end{align}
\begin{align}\label{EstimNorm4}
&\frac{1}{\pi}
\inf\limits_{\lambda\in \mathbb{R}}
\left \| \sum\limits_{k=n}^{\infty}e^{-\alpha k^{r}}\cos(kt+\xi) - \lambda \right \|_{p'}
= e^{-\alpha n^{r}}n^{\frac{1-r}{p}}  \left( \frac{\| \cos t\|_{p'} }{\pi^{1+\frac{1}{p'}}(\alpha r)^{\frac{1}{p}}} F^{\frac{1}{p'}}\left(\frac{1}{2}, \frac{3-p'}{2}; \frac{3}{2}; 1 \right)
\right.
\notag \\
+&
\left.
\bar{\gamma}_{n,p}^{(2)}
\left(\left(1+ \frac{(\alpha r)^{\frac{p'-1}{p}}}{p'-1} \right) 
 \frac{1}{n^{\frac{1-r}{p}}}
+
\frac{p^{\frac{1}{p'}}}{(\alpha r)^{1+\frac{1}{p}}n^{r}} 
  \right)\right),
\end{align}
де $\frac{1}{p}+\frac{1}{p'}=1$, а для величин $\bar{\gamma}_{n,p}^{(i)} = \bar{\gamma}_{n,p}^{(i)}(\alpha,r , \xi) $ виконуються нерівності 
$|\bar{\gamma}_{n,p}^{(i)}| \leq (14\pi)^{2}$.

Застосовуючи формулу \eqref{EstimNorm3} при $\xi=\gamma_{n}$, де $\gamma_{n}$ означена рівністю  \eqref{gamma_n},  і враховуючи, що $n_{0}(\alpha,r,p)\geq n_{*}(\alpha,r,p)$, із \eqref{rNAdditEstimate} при $n\geq n_{*}(\alpha,r,p)$, $1<p<\infty$, маємо

\begin{align}\label{RhoEstimate3}
|\tilde{\rho}_{n} (f;x)|\leq& 2 e^{-\alpha n^{r}}n^{\frac{1-r}{p}} \left| \sin \frac{2n-1}{2}x \right|
\left(
\frac{\|\cos t \|_{p'}}{\pi^{1+\frac{1}{p'}}(\alpha r)^{\frac{1}{p}}} F^{\frac{1}{p'}} \left(\frac{1}{2}, \frac{3-p'}{2}; \frac{3}{2}; 1 \right)\right.
\notag \\
+&
\left.
\left( \bar{\gamma}_{n,p}^{(1)}
\left(1+ \frac{(\alpha r)^{\frac{p'-1}{p}}}{p'-1}\right) +\theta_{n,p}\right) 
 \frac{1}{n^{\frac{1-r}{p}}}
+
\bar{\gamma}_{n,p}^{(1)}
\frac{p^{\frac{1}{p'}}}{(\alpha r)^{1+\frac{1}{p}}n^{r}} 
\right) E_{n}(f^{\alpha,r}_{\beta})_{L_{p}}.
\end{align}

Оскільки $|\bar{\gamma}_{n,p}^{(1)}+\theta_{n,p} | < 20\pi^{4}$, то з \eqref{RhoEstimate3} випливає оцінка \eqref{Theorem_Case1<p<infty}.

Далі доведемо  справедливість другої частини Теореми~\ref{theorem_1<p<infty}.

Користуючись інтегральним зображенням \eqref{IntegrRepr} та беручи до уваги ортогональність функції $r_{n}(t)$ вигляду \eqref{rn} до будь-якого тригонометричного полінома $t_{n}\in\mathcal{T}_{2n-1}$, для довільної функції $f\in C^{\alpha,r}_{\beta}L_{p}$, $1\leq p\leq\infty$, в кожній точці $x\in \mathbb{R}$ можемо записати

\begin{align}\label{IntegrRepr1}
&\tilde{\rho}_{n}(f;x)=
f(x)-\tilde{S}_{n}(f;x) \notag \\
=&2\sin\frac{2n-1}{2}x 
\left( \frac{1}{\pi}\int\limits_{-\pi}^{\pi}f^{\alpha,r}_{\beta}(t+x)
\sum\limits_{k=n}^{\infty}e^{-\alpha k^{r}}\cos (kt+\gamma_{n})
 dt
 +
 \frac{1}{\pi}\int\limits_{-\pi}^{\pi}\delta_{n}(t+x)
r_{n}(t)
 dt \right),
\end{align}
де  $\delta_{n}(\cdot)= f^{\alpha,r}_{\beta}(\cdot)-t_{n-1}(\cdot)$, $t_{n-1}$ --- довільний  поліном з $\mathcal{T}_{2n-1}$, а $r_{n}(t)$ i $\gamma_{n}=\gamma_{n}(\beta;x)$ означені за допомогою рівностей \eqref{rn} та \eqref{gamma_n} відповідно. Для функції 
\begin{equation}\label{g_x}
g_{x}(\cdot):=\frac{1}{\pi}\int\limits_{-\pi}^{\pi}f^{\alpha,r}_{\beta}(t+x)
\sum\limits_{k=1}^{\infty}e^{-\alpha k^{r}}\cos (kt+\gamma_{n})
 dt,
\end{equation}
яка очевидно належить до множини $C^{\alpha,r}_{2\gamma_{n}/\pi}L_{p}$, при фіксованому $x\in\mathbb{R}$  відхилення її частинних сум Фур'є $S_{n-1}(g_{x}, \cdot)$ порядку $n-1$ підпорядковані рівності
\begin{equation}\label{rho_n}
\rho(g_{x}; \cdot)=
g_{x}(\cdot)-S_{n-1}(g_{x}, \cdot)
=
\frac{1}{\pi}\int\limits_{-\pi}^{\pi}f^{\alpha,r}_{\beta}(t+\cdot)
\sum\limits_{k=n}^{\infty}e^{-\alpha k^{r}}\cos (kt+\gamma_{n})
 dt,
\end{equation}
і, зокрема
\begin{equation}\label{rho_nx}
\rho(g_{x}; x)=
g_{x}(x)-S_{n-1}(g_{x}, x)
=
\frac{1}{\pi}\int\limits_{-\pi}^{\pi}f^{\alpha,r}_{\beta}(t+x)
\sum\limits_{k=n}^{\infty}e^{-\alpha k^{r}}\cos (kt+\gamma_{n})
 dt.
\end{equation}

У відповідності з Теоремою 1 роботи \cite{SerdyukStepanyukFilomat} для функції $g_{x}(\cdot)$ при кожному $n\in\mathbb{N}$ знайдеться функція $G(\cdot)= G(f;n;x;\cdot)$ така, що
\begin{equation}\label{G_equality}
E_{n}(G^{\alpha,r}_{2\gamma_{n}/\pi})_{L_{p}}=
E_{n}(f^{\alpha,r}_{\beta})_{L_{p}}, \ \ 1<p<\infty,
\end{equation}
і для якої при $n\geq n_{0}(\alpha,r,p)$
\begin{align}\label{normC_G}
&\| {\rho}_{n}(G;\cdot)\|_{C}=
\| f(x)-{S}_{n}(f;x) \|_{C} \notag \\
=&e^{-\alpha n^{r}}n^{\frac{1-r}{p}}
\left(
\frac{\|\cos t \|_{p'}}{\pi^{1+\frac{1}{p'}}(\alpha r)^{\frac{1}{p}}}F^{\frac{1}{p'}}\left(\frac{1}{2},\frac{3-p'}{2};\frac{3}{2};1 \right) \right.
\notag \\
+&\left.
\gamma_{n,p}\left( 
\left(1+\frac{(\alpha r)^{\frac{p'-1}{p}}}{p'-1}\right)\frac{1}{n^{\frac{1-r}{p}}}
+\frac{p^{\frac{1}{p'}}}{(\alpha r)^{1+\frac{1}{p}}n^{r}}
\right)
\right) E_{n}(f^{\alpha,r}_{\beta})_{L_{p}}, \ \frac{1}{p}+\frac{1}{p'}=1,
\end{align}
де $\gamma_{n,p}=\gamma_{n,p}(\alpha,r,\beta,x)$ підпорядкована нерівності $|\gamma_{n,p}|\leq (14\pi)^{2}$.

Виберемо точку $x_{0}$ таким чином, щоб справджувалася рівність
\begin{equation}\label{G_equality1}
| {\rho}_{n}(G;x_{0})|=\| {\rho}_{n}(G;\cdot)\|_{C}.
\end{equation}

Покладемо
\begin{equation}\label{F_equality1}
\mathcal{F}(t):= \mathcal{J}^{\alpha,r}_{\beta} G^{\alpha,r}_{2\gamma_{n}/\pi}(t-x+x_{0}).
\end{equation}
За означенням $\mathcal{F}(t)\in C^{\alpha,r}_{\beta}L_{p}$.
Покажемо, що вона є шуканою функцією. Дійсно, оскільки згідно з \eqref{F_equality1} $\mathcal{F}^{\alpha,r}_{\beta}(t)= G^{\alpha,r}_{2\gamma_{n}/\pi}(t-x+x_{0})$, то з урахуванням \eqref{normC_G} та інваріантності $L_{p}$--норми відносно зсуву аргументу маємо
\begin{equation}\label{F_equality2}
E_{n}( \mathcal{F}^{\alpha,r}_{\beta})_{L_{p}}= E_{n}( G^{\alpha,r}_{2\gamma_{n}/\pi})_{L_{p}}=
E_{n}( f^{\alpha,r}_{\beta})_{L_{p}}, \ 1<p<\infty.
\end{equation}

Крім того, в силу \eqref{IntegrRepr1}, \eqref{G_equality}, \eqref{normC_G}, \eqref{G_equality1}, 
\eqref{rnEstimate} i  \eqref{AdditionEstimRho} для довільного заданого значення аргументу $x\in\mathbb{R}$ при $n\geq n_{*}(\alpha,r,p)$
\begin{align}\label{EstimateRho_F}
&| \tilde{\rho}_{n}(\mathcal{F};x)|\notag \\
=
&2\left| \sin\frac{2n-1}{2}x \right|
\left( \frac{1}{\pi} \left| \int\limits_{-\pi}^{\pi}G^{\alpha,r}_{2\gamma_{n}/\pi}(x_{0}+t)
\sum\limits_{k=n}^{\infty}e^{-\alpha k^{r}}\cos (kt+\gamma_{n})
 dt\right|
 +
\theta_{n,p}e^{-\alpha n^{r}}E_{n}( f^{\alpha,r}_{\beta})_{L_{p}} \right)
 \notag \\
 =
&2\left| \sin\frac{2n-1}{2}x \right|
\left( | {\rho}_{n}(G;x_{0})|
 +
\theta_{n,p}e^{-\alpha n^{r}}E_{n}( f^{\alpha,r}_{\beta})_{L_{p}} \right)
 \notag \\
 =
&2\left| \sin\frac{2n-1}{2}x \right|
\left( \| {\rho}_{n}(G; \cdot)\|_{C}
 +
\theta_{n,p}e^{-\alpha n^{r}}E_{n}( f^{\alpha,r}_{\beta})_{L_{p}} \right)
 \notag \\
=&2\left| \sin\frac{2n-1}{2}x \right| e^{-\alpha n^{r}}n^{\frac{1-r}{p}}
\left(
\frac{\|\cos t \|_{p'}}{\pi^{1+\frac{1}{p'}}(\alpha r)^{\frac{1}{p}}}F^{\frac{1}{p'}}\left(\frac{1}{2},\frac{3-p'}{2};\frac{3}{2};1 \right) \right.
\notag \\
+&\left.
\gamma_{n,p}\left( 
\left(1+\frac{(\alpha r)^{\frac{p'-1}{p}}}{p'-1}\right)\frac{1}{n^{\frac{1-r}{p}}}
+\frac{p^{\frac{1}{p'}}}{(\alpha r)^{1+\frac{1}{p}}n^{r}}
\right) + \frac{\theta_{n,p}}{n^{\frac{1-r}{p}}}
\right) E_{n}(f^{\alpha,r}_{\beta})_{L_{p}}, 
\end{align}
де для величин $\gamma_{n,p}$ i $\theta_{n,p}$ виконуються оцінки 
$|\gamma_{n,p}|\leq (14\pi)^{2}$, $| \theta_{n,p}|< \frac{1272}{169\pi}$ i $|\theta_{n,p}+ \gamma_{n,p}|<20\pi^{4}$. Із рівностей \eqref{EstimateRho_F} випливає \eqref{Theorem_Case1<p<inftyEquality}.
Теорему~\ref{theorem_1<p<infty} доведено.
\end{proof}

\begin{theorem}\label{theorem_p=1}
Нехай $0<r<1$, $\alpha>0$, $\beta\in\mathbb{R}$ i $n\in\mathbb{N}$. Тоді, для всякого $x\in\mathbb{R}$ i довільної функції
 $f\in C^{\alpha,r}_{\beta}L_{1}$ при $n\geq n_{*}(\alpha,r,1)$ виконується нерівність
 \begin{equation}\label{Theorem_Ineq1_p=1}
|\tilde{\rho}_{n}(f;x)|
\leq
2 e^{-\alpha n^{r}}n^{1-r}
\left|\sin \frac{2n-1}{2}x \right| 
\Big(
\frac{1}{\pi\alpha r}+\gamma_{n,1}^{*}\Big(\frac{1}{(\alpha r)^{2}}\frac{1}{n^{r}}+\frac{1}{n^{1-r}}\Big)\Big)E_{n}(f^{\alpha,r}_{\beta})_{L_{1}}.
 \end{equation}
 
 Крім того, для довільної функції  $f\in C^{\alpha,r}_{\beta}L_{1}$ можна вказати функцію ${\mathcal{F}(\cdot)=\mathcal{F}(f;n; x, \cdot)}$ з множини $C^{\alpha,r}_{\beta}L_{1}$ таку, що
  $E_{n}(\mathcal{F}^{\alpha,r}_{\beta})_{L_{1}}=E_{n}(f^{\alpha,r}_{\beta})_{L_{1}}$ 
і при  $n\geq n_{*}(\alpha,r,1)$  має місце рівність
 \begin{equation}\label{Theorem_Eq_p=1}
|\tilde{\rho}_{n}(\mathcal{F}; x)|
=
2 e^{-\alpha n^{r}}n^{1-r}
\left|\sin \frac{2n-1}{2}x \right| 
\Big(
\frac{1}{\pi\alpha r}+\gamma_{n,1}^{*}\Big(\frac{1}{(\alpha r)^{2}}\frac{1}{n^{r}}+\frac{1}{n^{1-r}}\Big)\Big)E_{n}(f^{\alpha,r}_{\beta})_{L_{1}}.
 \end{equation}  
В (\ref{Theorem_Ineq1_p=1})  і  (\ref{Theorem_Eq_p=1}) величини ${\gamma_{n,1}^{*}=\gamma_{n,1}^{*}(\alpha,r,\beta,f,x)}$ такі, що  ${|\gamma_{n,1}^{*}|<20\pi^{4}}$.
\end{theorem}

\begin{proof}[Доведення Теореми~\ref{theorem_p=1}]

Для доведення нерівності \eqref{Theorem_Ineq1_p=1} використаємо формулу \eqref{RhoEstimate2} при $p=1$, згідно з якою при довільних $x\in \mathbb{R}$, $f\in C^{\alpha,r}_{\beta}L_{1}$ i  $n\geq n_{*}(\alpha,r,1)$

 \begin{align}\label{rho_Theorem_p=1}
&|\tilde{\rho}_{n}(f; x)|
\leq
2 e^{-\alpha n^{r}}n^{1-r}
\left|\sin \frac{2n-1}{2}x \right| 
\Big(
\frac{\|\cos t\|_{\infty}}{\pi \alpha r} \left\| \frac{1}{\sqrt{t^{2}+1}} \right\|_{L_{\infty}[0, \frac{\pi n^{1-r}}{\alpha r}]}
\notag \\
+&
\frac{\gamma_{n,1}^{(1)}}{(\alpha r)^{2}}\frac{1}{n^{r}}  \left\| \frac{1}{\sqrt{t^{2}+1}} \right\|_{L_{\infty}[0, \frac{\pi n^{1-r}}{\alpha r}]}
+(\gamma_{n,1}^{(1)}+ \theta_{n,1}  ) \frac{1}{n^{1-r}}\Big)\Big)E_{n}(f^{\alpha,r}_{\beta})_{L_{1}} \notag\\
=&
2 e^{-\alpha n^{r}}n^{1-r}
\left|\sin \frac{2n-1}{2}x \right| 
\Big(\frac{1}{\pi\alpha r}+ \frac{\gamma_{n,1}^{(1)}}{(\alpha r)^{2}}\frac{1}{n^{r}}
+(\gamma_{n,1}^{(1)}+ \theta_{n,1}  ) \frac{1}{n^{1-r}}   \Big)E_{n}(f^{\alpha,r}_{\beta})_{L_{1}},
 \end{align}  
де $|\gamma_{n,1}|\leq (14\pi)^{2}$, $| \theta_{n,1}|< \frac{1272}{169\pi}$. Оскільки $|\gamma_{n,1}^{(1)}+ \theta_{n,1} | < 20\pi^{4}$, то із \eqref{rho_Theorem_p=1} випливає оцінка \eqref{Theorem_Ineq1_p=1}.

Доведемо другу частину Теореми~\ref{theorem_p=1}. Для довільної
 $f \in C^{\alpha,r}_{\beta}L_{1}$ і довільного фіксованого $x\in\mathbb{R}$ має місце рівність \eqref{IntegrRepr1}, в якій $f^{\alpha,r}_{\beta} \in L_{1}$.
Розглянемо функцію $g_{x}(\cdot)$ вигляду \eqref{g_x}, що очевидно належить до множини $C^{\alpha,r}_{2\gamma_{n}/\pi}L_{1}$. Для відхилень частинних сум Фур'є $S_{n-1}(g_{x};\cdot)$ порядку $n-1$ від функції $g_{x}(\cdot)$ виконується рівність \eqref{rho_n}, (а отже і  \eqref{rho_nx}). Видповідно до Теореми 2 роботи \cite{SerdyukStepanyukFilomat} для функції $g_{x}(\cdot)$ при кожному $n\in\mathbb{N}$ знайдеться функція $G(\cdot)= G(f,n;x;\cdot)$ така, що

\begin{equation}\label{G_derivative}
E_{n}(G^{\alpha,r}_{2\gamma_{n}/\pi})_{L_{1}}=E_{n}(f^{\alpha,r}_{\beta})_{L_{1}}
\end{equation}
і для якої при $n\geq n_{0}(\alpha,r,1)$
\begin{align}\label{rho_G_Theorem_p=1}
\| \rho_{n}(G,\cdot)\|_{C}=&
\| G(\cdot) - S_{n-1}(G; \cdot)\|_{C} \notag \\
&=e^{-\alpha n^{r}}n^{1-r} 
\Big(
\frac{1}{\pi\alpha r}+\gamma_{n,1}\Big(\frac{1}{(\alpha r)^{2}}\frac{1}{n^{r}}+\frac{1}{n^{1-r}}\Big)\Big)E_{n}(f^{\alpha,r}_{\beta})_{L_{1}},
\end{align}
де $\gamma_{n,1}=\gamma_{n,1}(\alpha,r,\beta, x)$ підпорядкована умові   $|\gamma_{n,1}|\leq (14\pi)^{2}$.

Виберемо точку $x_{0}$ таким чином, щоб виконувалась рівність \eqref{G_equality1}. Розглянемо функцію $\mathcal{F}(t)$, означену рівністю \eqref{F_equality1}, яка очевидно належить множині  $C^{\alpha,r}_{\beta}L_{1}$ і покажемо, що ця функція є шуканою функцією. Для функції  $\mathcal{F}(t)$, з урахуванням формули \eqref{G_derivative}, \eqref{F_equality1} та інваріантності $L_{1}$-норми відносно зсуву аргументу маємо
 \begin{equation}\label{G_derivative1}
E_{n}(\mathcal{F}^{\alpha,r}_{\beta})_{L_{1}}=E_{n}(G^{\alpha,r}_{2\gamma_{n}/\pi})_{L_{1}}=E_{n}(f^{\alpha,r}_{\beta})_{L_{1}}.
\end{equation}

Крім того, в силу \eqref{IntegrRepr1}, \eqref{G_derivative}, \eqref{rho_G_Theorem_p=1}, \eqref{G_equality1}, \eqref{rNAdditEstimate}, \eqref{HolderIneq},   для довільного заданого значення аргументу $x\in\mathbb{R}$ при $n\geq n_{*}(\alpha,r,1)$ 

\begin{align}\label{rho_F_Theorem_p=1}
& |\tilde{\rho}_{n}(\mathcal{F}; x)|
 \notag \\
 = &2\left| \sin\frac{2n-1}{2}x \right|
\left( \frac{1}{\pi}  \left|\int\limits_{-\pi}^{\pi}  G^{\alpha,r}_{2\gamma_{n}/\pi}(x_{0}+t)
\sum\limits_{k=n}^{\infty}e^{-\alpha k^{r}} \cos (kt+\gamma_{n}) 
 dt \right|
 +
\theta_{n,1}e^{-\alpha n^{r}} E_{n}( f^{\alpha,r}_{\beta})_{L_{1}} \right)
 \notag \\
  = &2\left| \sin\frac{2n-1}{2}x \right|
\left( |\rho_{n}(G,x_{0})|
 +
\theta_{n,1}e^{-\alpha n^{r}} E_{n}( f^{\alpha,r}_{\beta})_{L_{1}} \right)
 \notag \\
 = &2\left| \sin\frac{2n-1}{2}x \right|
\left(\|\rho_{n}(G,\cdot)\|_{C}
 +
\theta_{n,1}e^{-\alpha n^{r}} E_{n}( f^{\alpha,r}_{\beta})_{L_{1}} \right)
 \notag \\
=&2\left| \sin\frac{2n-1}{2}x \right|
e^{-\alpha n^{r}}n^{1-r} 
\Big(
\frac{1}{\pi\alpha r}+\gamma_{n,1}\Big(\frac{1}{(\alpha r)^{2}}\frac{1}{n^{r}}+\frac{1}{n^{1-r}}\Big)+  \frac{\theta_{n,1}}{n^{1-r}}\Big)E_{n}(\varphi)_{L_{1}},
\end{align}
де для величин $\gamma_{n,1}$ та $\theta_{n,1}$ виконуються оцінки $|\gamma_{n,1}|\leq (14\pi)^{2}$,  $|\theta_{n,1}|<\frac{1272}{169\pi}$ i $|\gamma_{n,1}+\theta_{n,1}|< 20\pi^{4}$. Із рівності \eqref{rho_F_Theorem_p=1} випливає \eqref{Theorem_Eq_p=1}. Теорему~\ref{theorem_p=1} доведено.
\end{proof}

\begin{theorem}\label{theorem_p=infty}
Нехай $r\in (0,1)$, $\alpha>0$  i   $\beta\in\mathbb{R}$.
Тоді для всіх $x\in\mathbb{R}$ i довільної функції $f\in C^{\alpha,r}_{\beta}L_{\infty}$ 
при $p=\infty$ i $n\geq n_{*}(\alpha,r,\infty)$  має місце
нерівність
\begin{equation}\label{Theorem_Ineq_p=infty}
|\tilde{\rho}_{n}(f;x)|
\leq
2e^{-\alpha n^{r}} \left|\sin \frac{2n-1}{2}x \right| 
\left(
\frac{4}{\pi^{2}}\ln \frac{n^{1-r}}{\alpha r}
+ \gamma^{*}_{n,\infty}
\right) E_{n}(f^{\alpha,r}_{\beta})_{L_{\infty}},
\end{equation}
де для величини  $ \gamma^{*}_{n,\infty}(\alpha,r,\beta,x) $ виконується оцінка $| \gamma^{*}_{n,\infty}|<20\pi^{4}$

Крім того, для довільної функції $f\in C^{\alpha,r}_{\beta}C $ можна вказати функцію  $\mathcal{F}(x)=\mathcal{F}(f;n;x; \cdot)$ з множини $C^{\alpha,r}_{\beta}C $ таку, що $E_{n}(\mathcal{F}^{\alpha,r}_{\beta})_{C}=E_{n}(f^{\alpha,r}_{\beta})_{C}$, і таку що для $n\geq n_{*}(\alpha,r,\infty)$ виконується рівність
\begin{equation}\label{Theorem_Eq_p=infty}
|\tilde{\rho}_{n}(\mathcal{F}; x)|=2 e^{-\alpha n^{r}}\left|\sin \frac{2n-1}{2}x \right| \left(\frac{4}{\pi^{2}}\ln \frac{n^{1-r}}{\alpha r} 
+\gamma_{n,\infty}^{**} \right) E_{n}(f^{\alpha,r}_{\beta})_{C}.
\end{equation}
У формулі  \eqref{Theorem_Eq_p=infty}
для величини $\gamma^{**}_{n,\infty}=\gamma^{**}_{n,\infty}(\alpha,r,\beta,f,x)$ виконується оцінка $|\gamma^{**}_{n,\infty}|<1951$.
\end{theorem}

 Зрозуміло, що якщо в умовах Теореми~\ref{theorem_p=infty} $f\in C^{\alpha,r}_{\beta}C$, то в нерівності \eqref{Theorem_Ineq_p=infty}  величину $E_{n}(f^{\alpha,r}_{\beta})_{L_{\infty}}$ можна замінити на $E_{n}(f^{\alpha,r}_{\beta})_{C}$.

\begin{proof}[Доведення Теореми~\ref{theorem_p=infty}]

Для доведення нерівності \eqref{Theorem_Ineq_p=infty} скористаємось оцінкою \eqref{RhoEstimate2} при $p=\infty$, а також наступною оцінкою (див. формулу (112) із \cite{SerdyukStepanyuk2018}):
\begin{equation}\label{I_estimate}
I_{1}\left(\frac{\pi n^{1-r}}{\alpha r} \right)
=
\int\limits_{0}^{\frac{\pi n^{1-r}}{\alpha r}}
\frac{dt}{\sqrt{1+t^{2}}}=
\ln \frac{\pi n^{1-r}}{\alpha r}+\Theta_{\alpha,r,n}, \ \ 0<\Theta_{\alpha,r,n}<1,
\end{equation}
де $I_{1}(v)$ означена формулою \eqref{Is_norm} при $s=p'=1$.

Отже, як випливає з  \eqref{RhoEstimate2} і \eqref{I_estimate}, при $n\geq n_{*}(\alpha,r,\infty)$ для довільної функції $f\in C^{\alpha,r}_{\beta}L_{\infty}$

\begin{align}\label{RhoEstimate4}
|\tilde{\rho}_{n} (f;x)|\leq& 2 e^{-\alpha n^{r}}  \left| \sin \frac{2n-1}{2}x \right|
\left(
\frac{4}{\pi^{2}}\ln \frac{\pi n^{1-r}}{\alpha r}
\right.
\notag \\
+&
\left.
\frac{4}{\pi^{2}}\Theta_{\alpha, r,n}
+ {\gamma}_{n,\infty}^{(1)}
\left( \ln \frac{\pi n^{1-r}}{\alpha r}+ \Theta_{\alpha,r,n}  \right) 
 \frac{1}{\alpha r n^{r}}
+
{\gamma}_{n,\infty}^{(1)}
+ \theta_{n,\infty}
\right) E_{n}(f^{\alpha,r}_{\beta})_{L_{\infty}}.
\end{align}

Оскільки при $n\geq n_{*}(\alpha,r,\infty)$ 

\begin{align}\label{FinalEstim}
&\frac{4}{\pi^{2}}\left( \ln \pi+ \Theta_{\alpha, r,n}\right) +
|{\gamma}_{n,\infty}^{(1)}|\left( \ln \frac{\pi n^{1-r}}{\alpha r}+
 \Theta_{\alpha,r,n}  \right)  \frac{1}{\alpha r n^{r}}+|{\gamma}_{n,\infty}^{(1)}
+ \theta_{n,\infty}|
\notag \\
<&
\frac{4(1+\ln\pi)}{\pi^{2}}+\frac{2(14\pi)^{2}}{(3\pi)^{3}}+(14\pi)^{2}+
\frac{1272}{169\pi}<20\pi^{4},
\end{align}
то з \eqref{RhoEstimate4} i \eqref{FinalEstim} випливає оцінка \eqref{Theorem_Ineq_p=infty}.

Доведемо  другу частину Теореми~\ref{theorem_p=infty}. Для довільної функції $f\in C^{\alpha,r}_{\beta}C$ і будь-якого фіксованого значення $x\in\mathbb{R}$ виконується рівність \eqref{IntegrRepr1}, в якій $f^{\alpha,r}_{\beta}\in C$. Розглянемо функцію $g_{x}(\cdot)$ вигляду \eqref{g_x} з множини $C^{\alpha, r}_{2\gamma_{n}/\pi}C$. Для відхилень $\rho_{n}(g_{x},\cdot)$ частинних сум Фур'є $S_{n-1}(g_{x};\cdot)$ порядку $n-1$ від функції $g_{x}(\cdot)$ виконується рівність \eqref{rho_n} (а отже і \eqref{rho_nx}). Відповідно до Теореми 1 роботи \cite{SerdyukStepanyukJAEN} для функції $g_{x}$ при будь-якому $n\in\mathbb{N}$ знайдеться функція $G(\cdot)= G(f;n;x;\cdot)$ з множини $C^{\alpha, r}_{2\gamma_{n}/\pi}C$ така, що
\begin{equation}\label{eqE_G}
E_{n}(G^{\alpha, r}_{2\gamma_{n}/\pi})_{C}= E_{n}(f^{\alpha,r}_{\beta})_{C}
\end{equation}
і для якої при всіх $n\geq n_{1}(\alpha,r)$, де $n_{1}(\alpha,r)$ --- найменше натуральне число, котре задовольняє нерівність
\begin{equation}\label{n1}
\frac{1}{\alpha r n^{r}} \left(  1+ \ln \frac{\pi n^{1-r}}{\alpha r} \right)+\frac{\alpha r}{n^{1-r}}  < \frac{1}{(3\pi)^{3}}, \ \ \alpha>0, \ r\in(0,1),
\end{equation}
виконуєтсья рівність
\begin{equation}\label{eq_rho_G}
\| \rho_{n}(G,\cdot) \|_{C}= \| G(\cdot) - S_{n-1}(G,\cdot)  \|_{C}
=e^{-\alpha n^{r}}\left( \frac{4}{\pi^{2}}  \ln \frac{ n^{1-r}}{\alpha r}  + \bar{\gamma}_{n,\infty}  \right)  E_{n}(f^{\alpha,r}_{\beta})_{C},
\end{equation}
 де $\bar{\gamma}_{n,\infty}=\bar{\gamma}_{n,\infty}(\alpha,r, \beta, x)$ підпорядкована умові $|\bar{\gamma}_{n,\infty} |\leq 20\pi^{4}$.
 
 Покажемо, що 
 \begin{equation}\label{Ineq_comparison_n}
  n_{1}(\alpha,r)\leq n_{*}(\alpha, r,\infty),
\end{equation} 
 тобто, що при будь-яких $\alpha>0$ i $r\in (0,\infty)$ із умови 
 \begin{equation}\label{additionalIneq}
  \frac{\ln \pi n}{\alpha r n^{r}} +\frac{\alpha r}{n^{1-r}}  \leq \frac{1}{(3\pi)^{3}}
\end{equation}
випливає нерівність \eqref{n1}. Дійсно, із \eqref{additionalIneq} безпосередньо отримуємо, що 
 \begin{equation*}
 \alpha r n^{r} \geq (3\pi)^{3} \ln \pi n \geq  (3\pi)^{3} \ln \pi,
\end{equation*}
а тому
 \begin{equation}\label{additionalIneq1}
 1- \ln \alpha r n^{r}  \leq 1- \ln(3\pi)^{3} \ln\pi <0.
\end{equation}

Отже, за виконання \eqref{additionalIneq}, з урахуванням \eqref{additionalIneq1} можемо записати
\begin{align*}
&\frac{1}{\alpha r n^{r}} \left(  1+ \ln \frac{\pi n^{1-r}}{\alpha r} \right)+\frac{\alpha r}{n^{1-r}} 
=\frac{1}{\alpha r n^{r}} \left(  1- \ln  \alpha r n^{r} + \ln \pi n \right)+\frac{\alpha r}{n^{1-r}} 
\notag \\
<&  \frac{\ln \pi n}{\alpha r n^{r}} +\frac{\alpha r}{n^{1-r}}  \leq \frac{1}{(3\pi)^{3}}
\end{align*} 
 звідки випливає \eqref{n1}. Тим самим нерівність \eqref{Ineq_comparison_n} доведено.

 Виберемо точку $x_{0}$ таким чином, щоб виконувалась рівність \eqref{G_equality1}. Розглянемо, як і раніше, функцію $\mathcal{F}$, що задається формулою \eqref{F_equality1}. Ця функція з множини $C^{\alpha,r}_{\beta}C$ і буде шуканою функцією. Для неї, з урахуванням \eqref{eqE_G} та інваріантності рівномірної норми відносно зсуву аргумента, маємо 
 \begin{equation}\label{eqE_GA}
E_{n}(\mathcal{F}^{\alpha,r}_{\beta})_{C}=E_{n}(G^{\alpha, r}_{2\gamma_{n}/\pi})_{C}= E_{n}(f^{\alpha,r}_{\beta})_{C}.
\end{equation}

 Крім того, в силу \eqref{IntegrRepr1},  \eqref{eqE_GA}, \eqref{eq_rho_G}, \eqref{G_equality1}, \eqref{rNAdditEstimate}, \eqref{bestApprox}, \eqref{HolderIneq} i \eqref{Ineq_comparison_n} для довільного значення аргумента $x\in\mathbb{R}$ при $n\geq  n_{*}(\alpha,r,\infty) $

\begin{align}\label{rho_F_Theorem_p=infty}
& |\tilde{\rho}_{n}(\mathcal{F}; x)|
 \notag \\
 = &2\left| \sin\frac{2n-1}{2}x \right|
\left( \frac{1}{\pi} \left| \int\limits_{-\pi}^{\pi}  G^{\alpha,r}_{2\gamma_{n}/\pi}(x_{0}+t)
\sum\limits_{k=n}^{\infty}e^{-\alpha k^{r}} \cos (kt+\gamma_{n}) 
 dt \right|
 +
\theta_{n,\infty}e^{-\alpha n^{r}} E_{n}( f^{\alpha,r}_{\beta})_{C} \right)
 \notag \\
 = & 2\left| \sin\frac{2n-1}{2}x \right|
\left(  \left| \rho_{n}(G;x_{0})\right|
 +
\theta_{n,\infty}e^{-\alpha n^{r}} E_{n}( f^{\alpha,r}_{\beta})_{C} \right)
 \notag \\
  = &
  2\left| \sin\frac{2n-1}{2}x \right|
\left(  \left \| \rho_{n}(G;\cdot)\right\|_{C}
 +
\theta_{n,\infty}e^{-\alpha n^{r}} E_{n}( f^{\alpha,r}_{\beta})_{C} \right)
 \notag \\
  = &
  2\left| \sin\frac{2n-1}{2}x \right| e^{-\alpha n^{r}}
\left( \frac{4}{\pi^{2}}\ln \frac{n^{1-r}}{\alpha r}+ \bar{\gamma}_{n,\infty}
 +
\theta_{n,\infty}  \right) E_{n}( f^{\alpha,r}_{\beta})_{C},
\end{align}
де для величин $ \bar{\gamma}_{n,\infty}$ i $\theta_{n,\infty}$ виконуються оцінки $|\bar{\gamma}_{n,\infty}| \leq 20\pi^{4}$, $|\theta_{n,\infty}|< \frac{1272}{169\pi}$ i $ |\bar{\gamma}_{n,\infty} + \theta_{n,\infty}| <1951$.
 Теорему~\ref{theorem_p=infty}  доведено.
\end{proof}

\section{Розв'язок задачі Колмогорова-Нікольського для інтерполяційних поліномів Лагранжа на класах узагальнених інтегралів Пуассона $C^{\alpha,r}_{\beta,p}$ }

Із Теорем \ref{theorem_1<p<infty}, \ref{theorem_p=1} та \ref{theorem_p=infty} даної роботи випливає, що нерівності \eqref{Theorem_Case1<p<infty}, \eqref{Theorem_Ineq1_p=1} та \eqref{Theorem_Ineq_p=infty} є асимптотично непокращуваними на множинах узагальнених інтегралів Пуассона $C^{\alpha,r}_{\beta}L_{p}$ при всіх $x\in\mathbb{R}$, $\beta\in\mathbb{R}$, $\alpha>0$, $r\in(0,1)$ i $1\leq p\leq \infty$. Зрозуміло, що зазначені нерівності мають місце і для довільних підмножин із $C^{\alpha,r}_{\beta}L_{p}$, якими є множини  $C^{\alpha,r}_{\beta}\mathfrak{N}$, $\mathfrak{N} \subset L_{p}$ i, зокрема класи 
 \begin{equation*}
C^{\alpha,r}_{\beta,p}=C^{\alpha,r}_{\beta}U_{p}, \ \ U_{p}=\left\{ \varphi \in L_{p}: \|\varphi\|_{p}\leq 1 \right\}.
\end{equation*}

Розглядаючи точні верхні межі в обох частинах кожної з нерівностей \eqref{Theorem_Case1<p<infty}, \eqref{Theorem_Ineq1_p=1} та \eqref{Theorem_Ineq_p=infty} по класу  $C^{\alpha,r}_{\beta,p}$ при $p\in(1,\infty)$, $p=1$ та $p=\infty$ відповідно, і врахувавши, що для довільної   $f\in C^{\alpha,r}_{\beta,p}$, $1\leq p\leq \infty$, $E_{n}( f^{\alpha,r}_{\beta})_{L_{p}}\leq 1$, отримуємо, що при $n\geq n_{*}(\alpha,r,p)$ і всіх $x\in \mathbb{R}$ вірні наступні нерівності:
\begin{align}\label{Addit_InterpInequality_1<p<infty}
\tilde{\mathcal{E}}_{n}(C^{\alpha,r}_{\beta,p};x)\leq & 2e^{-\alpha n^{r}}n^{\frac{1-r}{p}}\left|\sin \frac{2n-1}{2}x \right| 
\left(
\frac{\|\cos t \|_{p'}}{\pi^{1+\frac{1}{p'}}(\alpha r)^{\frac{1}{p}}}F^{\frac{1}{p'}}\left(\frac{1}{2},\frac{3-p'}{2};\frac{3}{2};1 \right) \right.
\notag \\
+&\left.
\bar{\gamma}^{*}_{n,p}\left( 
\left(1+\frac{(\alpha r)^{\frac{p'-1}{p}}}{p'-1}\right)\frac{1}{n^{\frac{1-r}{p}}}
+\frac{p^{\frac{1}{p'}}}{(\alpha r)^{1+\frac{1}{p}}n^{r}}
\right)
\right),  \ 1<p<\infty, \ \frac{1}{p}+\frac{1}{p'}=1,
\end{align}
\begin{align}\label{Addit_InterpInequality_p=1}
\tilde{\mathcal{E}}_{n}(C^{\alpha,r}_{\beta,1};x)
\leq 2e^{-\alpha n^{r}}n^{1-r}\left|\sin \frac{2n-1}{2}x \right| 
\left(
\frac{1}{\pi\alpha r}+ \bar{\gamma}^{*}_{n,1}\left(\frac{1}{(\alpha r)^{2}n^{r}}+
\frac{1}{n^{1-r}}
\right)
\right),  
\end{align}
\begin{align}\label{Addit_InterpInequality_p=infty}
\tilde{\mathcal{E}}_{n}(C^{\alpha,r}_{\beta, \infty};x)
\leq
2e^{-\alpha n^{r}} \left|\sin \frac{2n-1}{2}x \right| 
\left(
\frac{4}{\pi^{2}}\ln \frac{n^{1-r}}{\alpha r}
+ \bar{\gamma}^{*}_{n,\infty}
\right),
\end{align}
 де $|\bar{\gamma}^{*}_{n,p}|<20\pi^{4}$ при $1\leq p\leq \infty$.

Виявляється, що в \eqref{Addit_InterpInequality_1<p<infty}--\eqref{Addit_InterpInequality_p=infty} знак "$\leq$"  можна замінити на знак "$=$". Цей факт випливатиме із наступного твердження.

\begin{theorem}\label{theorem_sup_Interpolation}
Нехай $r\in(0,1)$, $\alpha>0$, $\beta\in\mathbb{R}$,  $1\leq p\leq \infty$ i $x\in\mathbb{R}$.
Тоді при $n\geq n_{*}(\alpha,r,1)$
\begin{align}\label{Theorem_sup_p=1}
\tilde{\mathcal{E}}_{n}(C^{\alpha,r}_{\beta,1};x)
=e^{-\alpha n^{r}}n^{1-r}\left|\sin \frac{2n-1}{2}x \right| 
\left(
\frac{2}{\pi\alpha r}+ \delta^{*}_{n,1}\left(
\frac{1}{n^{1-r}}+
\frac{1}{(\alpha r)^{2}n^{r}}
\right)
\right),  
\end{align}
при   $n\geq n_{*}(\alpha,r,p)$ i $1<p<\infty$ 

\begin{align}\label{Theorem_sup_1<p<infty}
&\tilde{\mathcal{E}}_{n}(C^{\alpha,r}_{\beta,p};x)=e^{-\alpha n^{r}}n^{\frac{1-r}{p}}\left|\sin \frac{2n-1}{2}x \right| 
\notag \\
\times&
\left(
\frac{2\|\cos t \|_{p'}}{\pi^{1+\frac{1}{p'}}(\alpha r)^{\frac{1}{p}}}F^{\frac{1}{p'}}\left(\frac{1}{2},\frac{3-p'}{2};\frac{3}{2};1 \right)+
\delta^{*}_{n,p}\left( 
\left(1+\frac{(\alpha r)^{\frac{p'-1}{p}}}{p'-1}\right)\frac{1}{n^{\frac{1-r}{p}}}
+\frac{p^{\frac{1}{p'}}}{(\alpha r)^{1+\frac{1}{p}}n^{r}}
\right)
\right), 
\end{align}
а при  $n\geq n_{*}(\alpha,r,\infty)$ 

\begin{align}\label{Theorem_sup_p=infty}
\tilde{\mathcal{E}}_{n}(C^{\alpha,r}_{\beta, \infty};x)
=
e^{-\alpha n^{r}} \left|\sin \frac{2n-1}{2}x \right| 
\left(
\frac{8}{\pi^{2}}\ln \frac{n^{1-r}}{\alpha r}
+ \delta^{*}_{n,\infty}
\right).
\end{align}
У формулах \eqref{Theorem_sup_p=1}--\eqref{Theorem_sup_p=infty}
для величин $\delta^{*}_{n,p}=\delta^{*}_{n,p}(\alpha,r,\beta,x)$ виконується оцінка $|\delta^{*}_{n,p}|<40\pi^{4}$.

\end{theorem}

\begin{proof}[Доведення Теореми~\ref{theorem_sup_Interpolation}]

Будемо відштовхуватись від інтегрального зображення \eqref{IntegrRepr}, у якому $f\in C^{\alpha,r}_{\beta}$, $1\leq p\leq \infty$. Розглянувши точні верхні межі модулів обох частин рівності  \eqref{IntegrRepr} при $t_{n-1} \equiv 0$ по класу $C^{\alpha,r}_{\beta,p}$ та врахувавши  інваріантність множини 
\begin{equation*}
U_{p}^{0}= \left\{\varphi\in L_{p}: \ \|\varphi\|_{p}\leq 1, \ \int\limits_{-\pi}^{\pi}\varphi(t)dt=0 \right\}, \ \ 1\leq p \leq\infty,
\end{equation*}
відносно зсуву аргументу при всіх $\alpha>0$, $\beta\in\mathbb{R}$, $x\in\mathbb{R}$, $1\leq p\leq\infty$ i $n\in\mathbb{N}$ будемо мати
\begin{align}\label{Equation_InProof_Th_Interp1}
&\tilde{\mathcal{E}}_{n}(C^{\alpha,r}_{\beta, p};x)
=
\sup\limits_{f\in C^{\alpha,r}_{\beta, p}}   |\tilde{\rho}_{n}(f;x)|
\notag \\
=&
2  \left|\sin \frac{2n-1}{2}x \right| 
\left(
\sup\limits_{\varphi \in U_{p}^{0}} \frac{1}{\pi}\int\limits_{-\pi}^{\pi} \varphi(t) \sum\limits_{k=n}^{\infty}e^{-\alpha k^{r}}\cos(kt+\gamma_{n}) dt+ \xi_{n} \|r_{n}(\cdot)\|_{C}
\right),
\end{align}
де $\gamma_{n}$ i $r_{n}(t)$ означені рівностями \eqref{gamma_n} i \eqref{rn} відповідно, а для величини $\xi_{n}=\xi_{n}(\alpha,r,\beta,p)$ виконується нерівність $|\xi_{n}|\leq 2$.

Із співвідношень двоїстості (див., наприклад, \cite[с. 27]{Korn}) маємо
\begin{align}\label{Equation_InProof_Th_Interp2}
\sup\limits_{\varphi \in U_{p}^{0}} \frac{1}{\pi}\int\limits_{-\pi}^{\pi} \varphi(t) \sum\limits_{k=n}^{\infty}e^{-\alpha k^{r}}\cos(kt+\gamma_{n}) dt=
\frac{1}{\pi}\inf\limits_{\lambda \in\mathbb{R}} \left\| \sum\limits_{k=n}^{\infty}e^{-\alpha k^{r}}\cos(kt+\gamma_{n}) -\lambda \right\|_{p'}, \ \frac{1}{p}+\frac{1}{p'}=1.
\end{align}

Розглянемо випадок $p=1$. Із \cite{SerdyukStepanyukDopov}--\cite{SerdyukStepanyuk2017} випливає, що при довільних $r\in(0,1)$, $\alpha>0$, $\xi\in\mathbb{R}$ i $n\geq n_{0}(\alpha,r,1)$, де  $n_{0}(\alpha,r,1)$ --- найменший з номерів $n$ такий, що задовольняє нерівність \eqref{n_0} при $p=1$,  мають місце оцінки
\begin{align}\label{Equation_InProof_Th_Interp3}
\left\| \sum\limits_{k=n}^{\infty}e^{-\alpha k^{r}}\cos(kt+\xi)  \right\|_{\infty}
= e^{-\alpha n^{r}} n^{1-r} \left(\frac{1}{\pi \alpha r} +\bar{\gamma}^{(1)}_{n,1}
\left(\frac{1}{(\alpha r)^{2}}\frac{1}{n^{r}}+\frac{1}{n^{1-r}} \right)
\right),
\end{align}
\begin{align}\label{Equation_InProof_Th_Interp4}
\frac{1}{\pi}\inf\limits_{\lambda \in\mathbb{R}} \left\| \sum\limits_{k=n}^{\infty}e^{-\alpha k^{r}}\cos(kt+\xi) -\lambda \right\|_{p'}
= e^{-\alpha n^{r}} n^{1-r} \left(\frac{1}{\pi \alpha r} +\bar{\gamma}^{(2)}_{n,1}
\left(\frac{1}{(\alpha r)^{2}}\frac{1}{n^{r}}+\frac{1}{n^{1-r}} \right)
\right),
\end{align}
в яких для величин $\bar{\gamma}^{(i)}_{n,1}= \bar{\gamma}^{(i)}_{n,1}(\alpha,r,\xi)$, $i=1,2,$ виконуються нерівності $|\bar{\gamma}^{(i)}_{n,1}| \leq (14\pi)^{2}$.

Застосувавши формулу \eqref{Equation_InProof_Th_Interp4} при $\xi=\gamma_{n}$, де $\gamma_{n}$ означена рівністю \eqref{gamma_n} і враховуючи, що при  $n_{0}(\alpha,r,1)\leq n_{*}(\alpha,r,1)$ із \eqref{Equation_InProof_Th_Interp1}, \eqref{Equation_InProof_Th_Interp2} i \eqref{Equation_InProof_Th_Interp4} та \eqref{rNAdditEstimate} для  $n\geq n_{*}(\alpha,r,1)$ маємо
\begin{align}\label{Equation_InProof_Th_Interp5}
\tilde{\mathcal{E}}_{n}(C^{\alpha,r}_{\beta,1};x)
=
& 2  \left|\sin \frac{2n-1}{2}x \right| 
\left(
e^{-\alpha n^{r}} n^{1-r} \left(\frac{1}{\pi \alpha r} +\bar{\gamma}^{(2)}_{n,1}
\left(\frac{1}{(\alpha r)^{2}}\frac{1}{n^{r}}+\frac{1}{n^{1-r}} \right)+
\theta_{n,1}e^{-\alpha n^{r}}
\right)
\right)
\notag \\
 =&
e^{-\alpha n^{r}} n^{1-r}  \left |\sin \frac{2n-1}{2}x \right| 
\left(
 \frac{2}{\pi \alpha r} +2\bar{\gamma}^{(2)}_{n,1}
\frac{1}{(\alpha r)^{2}}\frac{1}{n^{r}}
+2 \left(\bar{\gamma}^{(2)}_{n,1} +  \theta_{n,1}\right)  \frac{1}{n^{1-r}}
\right),
\end{align}
де $|\theta_{n,1}| < \frac{1272}{169\pi}$.

Оскільки $2 \left(|\bar{\gamma}^{(2)}_{n,1}| + | \theta_{n,1}| \right)<40\pi^{4}$, то із \eqref{Equation_InProof_Th_Interp5} випливає \eqref{Theorem_sup_p=1}.

Нехай $p\in (0,1)$. В цьому випадку має місце оцінка \eqref{EstimNorm4}, що виконується при всіх $n\geq n_{0}(\alpha,r,p)$, $1<p<\infty$. Застосувавши \eqref{EstimNorm4} при $\xi=\gamma_{n}$, де $\gamma_{n}$ означена формулою \eqref{gamma_n} і врахувавши, що $n_{0}(\alpha,r,p) \leq n_{*}(\alpha,r,p)$, $1<p<\infty$,  із \eqref{Equation_InProof_Th_Interp1}, \eqref{Equation_InProof_Th_Interp2} та \eqref{rNAdditEstimate} при $n\geq n_{*}(\alpha,r,p)$ маємо

\begin{align}\label{Equation_InProof_Th_Interp6}
&\tilde{\mathcal{E}}_{n}(C^{\alpha,r}_{\beta, p};x)
=
2 \left|\sin \frac{2n-1}{2}x \right| 
  \left( e^{-\alpha n^{r}}n^{\frac{1-r}{p}}\left(
\frac{\|\cos t \|_{p'}}{\pi^{1+\frac{1}{p'}}(\alpha r)^{\frac{1}{p}}}F^{\frac{1}{p'}}\left(\frac{1}{2},\frac{3-p'}{2};\frac{3}{2};1 \right) \right. \right.
\notag \\
+&\left. \left.
\bar{\gamma}^{(2)}_{n,p}\left( 
\left(1+\frac{(\alpha r)^{\frac{p'-1}{p}}}{p'-1}\right)\frac{1}{n^{\frac{1-r}{p}}}
+\frac{p^{\frac{1}{p'}}}{(\alpha r)^{1+\frac{1}{p}}n^{r}}
\right)  \right) +\theta_{n,p}e^{-\alpha n^{r}}
\right) \notag\\
=&
 e^{-\alpha n^{r}}n^{\frac{1-r}{p}}
  \left|\sin \frac{2n-1}{2}x \right| 
  \left(
\frac{2\|\cos t \|_{p'}}{\pi^{1+\frac{1}{p'}}(\alpha r)^{\frac{1}{p}}}F^{\frac{1}{p'}}\left(\frac{1}{2},\frac{3-p'}{2};\frac{3}{2};1 \right) \right.
\notag \\
+&\left.
2\bar{\gamma}^{(2)}_{n,p}
\frac{p^{\frac{1}{p'}}}{(\alpha r)^{1+\frac{1}{p}}}
\frac{1}{n^{r}}
+
\left( 2\bar{\gamma}^{(2)}_{n,p}
\left(1+\frac{(\alpha r)^{\frac{p'-1}{p}}}{p'-1}\right)  +2\theta_{n,p} \right)       \frac{1}{n^{\frac{1-r}{p}}}
\right),
\end{align}
де $|\theta_{n,p}| < \frac{1272}{169\pi}$.

Враховуючи, що  $2 \left(|\bar{\gamma}^{(2)}_{n, p}| + | \theta_{n,p}| \right)<40\pi^{4}$,  із \eqref{Equation_InProof_Th_Interp6} отримуємо \eqref{Theorem_sup_1<p<infty}.

Нехай, нарешті, $p=\infty$. 

В силу \eqref{I_estimate} та формул \eqref{EstimNorm3} i \eqref{EstimNorm4}, застосованих при $p=\infty$, для усіх номерів $n\geq n_{0}(\alpha,r,\infty)$ і довільних $\alpha>0$, $r\in(0,1)$, $\gamma\in \mathbb{R}$, мають місце оцінки
\begin{align}\label{norm_p=infty_1}
&\frac{1}{\pi}\left \| \sum\limits_{k=n}^{\infty}e^{-\alpha k^{r}}\cos(kt+\xi) \right \|_{1} \notag \\
  =& e^{-\alpha n^{r}}  \left( \frac{4}{\pi^{2}} \ln \frac{\pi n^{1-r}}{\alpha r}
 + \frac{4}{\pi^{2}} \Theta_{\alpha,r,n}
+\gamma_{n,\infty}^{(1)} 
 \left( \ln \frac{\pi n^{1-r}}{\alpha r}
 +  \Theta_{\alpha,r,n} \right) \frac{1}{\alpha r n^{r}}
 +
\gamma_{n,\infty}^{(1)}  
\right),
\end{align}

\begin{align}\label{norm_p=infty_2}
&\frac{1}{\pi}
\inf\limits_{\lambda\in \mathbb{R}}
\left \| \sum\limits_{k=n}^{\infty}e^{-\alpha k^{r}}\cos(kt+\xi) - \lambda \right \|_{1} 
\notag \\  
  =& e^{-\alpha n^{r}}  \left( \frac{4}{\pi^{2}} \ln \frac{\pi n^{1-r}}{\alpha r}
 + \frac{4}{\pi^{2}} \Theta_{\alpha,r,n}
+
\gamma_{n,\infty}^{(2)}  \left( \ln \frac{\pi n^{1-r}}{\alpha r}
 +  \Theta_{\alpha,r,n} \right) \frac{1}{\alpha r n^{r}}
 +
\gamma_{n,\infty}^{(2)} 
\right),
\end{align}
де $|\gamma_{n,\infty}^{(i)}|\leq (14\pi)^{2}$, $i=1,2,$ а $0<\Theta_{\alpha,r,n}<1$.

При $n\geq n_{1}(\alpha,r)$, де $n_{1}(\alpha,r)$ --- найменше натуральне число, яке задовольняє нерівність \eqref{n1}, мають місце оцінки 
\begin{align}\label{Equation_InProof_Th_Interp6}
 &\frac{4}{\pi^{2}} \left(\ln \pi + \Theta_{\alpha,r,n} \right) +|\gamma_{n,\infty}^{(i)}|  
 \left( \ln \frac{\pi n^{1-r}}{\alpha r}
 +  \Theta_{\alpha,r,n} \right) \frac{1}{\alpha r n^{r}} +|\gamma_{n,\infty}^{(i)}|  
 \notag \\
 &<
 \frac{4(1+\ln\pi)}{\pi^{2}} + (14\pi)^{2}\left(\frac{1}{(3\pi)^{3}}+1 \right)<1938.
\end{align}

Із \eqref{norm_p=infty_1}--\eqref{norm_p=infty_2} з урахуванням нерівності $n_{1}(\alpha,r)> n_{0}(\alpha,r,\infty)$, отримуємо, що при $\alpha>0$, $r\in(0,1)$, $\gamma \in \mathbb{R}$ i $n\geq n_{1}(\alpha,r)$
\begin{align}\label{norm_p=infty_3}
\frac{1}{\pi}
\left \| \sum\limits_{k=n}^{\infty}e^{-\alpha k^{r}}\cos(kt+\xi)  \right \|_{1} 
= e^{-\alpha n^{r}}  \left( \frac{4}{\pi^{2}} \ln \frac{n^{1-r}}{\alpha r}
+
\bar{\gamma}_{n,\infty}^{(1)}  
\right),
\end{align}
\begin{align}\label{norm_p=infty_4}
\frac{1}{\pi}
\inf\limits_{\lambda\in \mathbb{R}}
\left \| \sum\limits_{k=n}^{\infty}e^{-\alpha k^{r}}\cos(kt+\xi) - \lambda \right \|_{1} 
= e^{-\alpha n^{r}}  \left( \frac{4}{\pi^{2}} \ln \frac{ n^{1-r}}{\alpha r}
+
\bar{\gamma}_{n,\infty}^{(2)}  
\right),
\end{align}
де $|\bar{\gamma}_{n,\infty}^{(i)}  |<1938$, $i=1,2$.

Застосувавши формулу \eqref{norm_p=infty_4} при $\xi=\gamma_{n}$, де $\gamma_{n}$ означена рівністю \eqref{gamma_n} і врахувавши \eqref{Ineq_comparison_n}, із \eqref{Equation_InProof_Th_Interp1}, \eqref{Equation_InProof_Th_Interp2} та \eqref{rNAdditEstimate} для $n\geq n_{*}(\alpha,r, \infty)$ маємо
\begin{align}\label{Equation_InProof_Th_Interp7}
&\tilde{\mathcal{E}}_{n}(C^{\alpha,r}_{\beta, \infty};x)
=
2 \left|\sin \frac{2n-1}{2}x \right| 
  \left( e^{-\alpha n^{r}}  \left( \frac{4}{\pi^{2}} \ln \frac{ n^{1-r}}{\alpha r}
+
\bar{\gamma}_{n,\infty}^{(2)}  
\right)+ \theta_{n,\infty}e^{-\alpha n^{r}} 
\right)
\notag \\
=&
 e^{-\alpha n^{r}} \left|\sin \frac{2n-1}{2}x \right| 
  \left(    \frac{8}{\pi^{2}} \ln \frac{ n^{1-r}}{\alpha r}
+
2\left(\bar{\gamma}_{n,\infty}^{(2)}  
+ \theta_{n,\infty}\right) 
\right),
\end{align}
де $|\theta_{n,\infty}| < \frac{1272}{169\pi}$.

Врахувавши, що  $2 \left(|\bar{\gamma}^{(2)}_{n, \infty}| + | \theta_{n,\infty}| \right)<40\pi^{4}$, то із \eqref{Equation_InProof_Th_Interp7} отримуємо \eqref{Theorem_sup_p=infty}.
Теорему~\ref{theorem_sup_Interpolation} доведено.
\end{proof}

Зауважимо, що оцінка \eqref{Theorem_sup_p=infty} Теореми~\ref{theorem_sup_Interpolation} уточнює оцінку (45)  роботи \cite{Serdyuk2004}.

Також зазначимо, що оцінки  \eqref{Theorem_sup_p=1}--\eqref{Theorem_sup_p=infty}, з яких випливають асимптотичні рівності величин $ \widetilde{{\cal E}}_{n}(C^{\alpha,r}_{\beta,p};x)$, $1\leq p\leq \infty$, $\alpha>0$, $\beta\in\mathbb{R}$, $r\in(0,1)$ при $n\rightarrow \infty$, є інтерполяційними аналогами, отриманих в  \cite{SerdyukStepanyuk2018} (див. також \cite{SerdyukStepanyukDopov}, \cite{SerdyukStepanyuk2017}, \cite{SerdyukStepanyuk_Interpolation2023}) відповідних оцінок для величин
\begin{equation}\label{FourierSums_definition}
{\mathcal{E}}_{n}(C^{\alpha,r}_{\beta,p})_{C}=\sup\limits_{f\in C^{\alpha,r}_{\beta,p} } 
\left\| f- S_{n-1}(f)\right\|_{C},
\end{equation}
де $ S_{n-1}(f)$ --- частинні суми Фур'є порядку $n-1$ функції $f$.

Співставивши формули  \eqref{Theorem_sup_p=1}--\eqref{Theorem_sup_p=infty} даної роботи з оцінками (17), (18), (30) роботи \cite{SerdyukStepanyuk2018}, переконуємось у виконанні граничного співвідношення для величин \eqref{quantityInterpol} i \eqref{FourierSums_definition} при $n\rightarrow \infty$
\begin{equation}\label{Limit_Relation1}
\lim\limits_{n\rightarrow\infty}\frac{\widetilde{{\cal E}}_{n}(C^{\alpha,r}_{\beta,p};x)}{\left|\sin \frac{2n-1}{2}x \right| {\mathcal{E}}_{n}(C^{\alpha,r}_{\beta,p})_{C}  }=2, \ 1\leq p\leq\infty, \ \alpha>0, \ r\in(0,1), \ \beta\in\mathbb{R}.
\end{equation}

Для класів функцій скінченної гладкості ситуація принципово інша. 

Нехай, як і раніше, $W^{r}_{\infty}$ --- клас $2\pi$--періодичних функцій $f$ таких, що їх похідні $f^{(k)}$ до $(r-1)$  порядку включно абсолютно неперервні, а $\| f^{(r)}\|_{\infty}\leq1$.
Як показав А.М. Колмогоров \cite{Kolmogorov1935}
\begin{equation}\label{KolmogorovAsymp}
{\mathcal{E}}_{n}(W^{r}_{\infty})_{C}= 
\sup\limits_{f\in W^{r}_{\infty} } 
\left\| f- S_{n-1}(f)\right\|_{C}=
\frac{4}{\pi^{2}}\frac{\ln n}{n^{r}} +
\mathcal{O}\left(\frac{1}{n^{r}}\right), \ \ r\in\mathbb{N}.
\end{equation}

 Співставлення оцінки Колмогорова \eqref{KolmogorovAsymp} для сум Фур'є з оцінкою Нікольського \eqref{Nikolsky1945} для інтерполяційних поліномів дозволяє записати при $x\neq \frac{2k\pi}{2n-1}$, $k\in \mathbb{Z}$, граничне співвідношення
  \begin{equation}\label{Limit_Relation2_SobolevClasses}
\lim\limits_{n\rightarrow\infty}\frac{\widetilde{{\cal E}}_{n}(W^{r}_{\infty};x)}{\left|\sin \frac{2n-1}{2}x \right| {\mathcal{E}}_{n}(W^{r}_{\infty})_{C}  }=2\sum\limits_{v=0}^{\infty}\frac{(-1)^{v(r+1)}}{(2v+1)^{r+1}}, \ r\in\mathbb{N}.
\end{equation}

Як видно із \eqref{Limit_Relation2_SobolevClasses}  границя відношення  $\frac{\widetilde{{\cal E}}_{n}(W^{r}_{\infty};x)}{\left|\sin \frac{2n-1}{2}x \right| {\mathcal{E}}_{n}(W^{r}_{\infty})_{C}  })$  при $n\rightarrow\infty$ залежить від показника гладкості.

Стосовно уточнень та узагальнень оцінки \eqref{KolmogorovAsymp} див., наприклад, \cite{Stepanets1}, \cite{Stechkin1980}, \cite{SerdyukSokolenko2022}, \cite{Teljakovsky1989} та ін.

\section{Доведення Леми~\ref{Lemma1}}\label{LemmaProof}

\begin{proof}[Доведення Леми~\ref{Lemma1}]

Нехай $n\in\mathbb{N}$ такий, що виконується нерівність \eqref{NumberIneq}.
 Покажемо, що при вказаних $n$ має місце оцінка
\begin{equation}\label{Lemma1Estim1}
\sum\limits_{k=1}^{\infty}\sum\limits_{v=(2k+1)n-k}^{\infty}e^{-\alpha v^{r}}
<
\frac{636}{169}\frac{n^{1-r}}{\alpha r } e^{-\alpha (3n-1)^{r}}.
\end{equation}

Перш за все зауважимо, що  для довільної додатньої і спадної функції $\xi(u)$ такої, що $\int\limits_{m}^{\infty}\xi(u) du<\infty$ має місце співвідношення
\begin{equation}\label{XiInequality}
\sum\limits_{j=m}^{\infty}\xi(j)<\xi(m)+\int\limits_{m}^{\infty}\xi(u) du.
\end{equation}
У роботі \cite{SerdyukStepanyuk2018} (формула (22)) була встановлена наступна оцінка:
\begin{equation}\label{Lemma1Estim2}
\int\limits_{m}^{\infty}e^{-\alpha t^{r}}t^{\delta}dt
=\frac{e^{-\alpha m^{r}}}{\alpha r}m^{\delta+1-r}
\left(1+
\Theta_{\alpha,m}^{r,\delta}\frac{|\delta+1-r|}{\alpha r}\frac{1}{m^{r}}
\right), \ \ 
|\Theta_{\alpha,m}^{r,\delta}|\leq \frac{14}{13},
\end{equation}
яка виконується при всіх $\alpha>0$, $r>0$, $\delta\in\mathbb{R}$ i номерах $m$ таких, що $m\geq \left(\frac{14|\delta+1-r|}{\alpha r}\right)^{\frac{1}{r}}$.

В силу \eqref{XiInequality}, \eqref{Lemma1Estim2} і умови \eqref{NumberIneq},
\begin{align}\label{Lemma1Estim3}
&\sum\limits_{v=(2k+1)n-k}^{\infty}e^{-\alpha v^{r}}
<
e^{-\alpha ( (2k+1)n-k )^{r}}+\int\limits_{(2k+1)n-k}^{\infty} 
e^{-\alpha t^{r}}  dt
\notag \\
\leq&
\frac{e^{-\alpha ( (2k+1)n-k )^{r}}  }{\alpha r} ((2k+1)n-k )^{1-r}
\left( \frac{\alpha r}{ ((2k+1)n-k )^{1-r}}+1+\frac{14}{13}\frac{1-r}{ \alpha r((2k+1)n-k )^{r}}
\right) \notag \\
<&\frac{14}{13} \frac{e^{-\alpha ( (2k+1)n-k )^{r}}  }{\alpha r}(2k+1)n-k )^{1-r}.
\end{align}

В силу  \eqref{NumberIneq}  і монотонного спадання функції $e^{-\alpha t^{r}} t^{1-r}$  на $[n,\infty)$ при  $n$, що задовольняють умову \eqref{NumberIneq}, з урахуванням \eqref{Lemma1Estim3} та рівності \eqref{Lemma1Estim2}, застосованої при $\delta=1-r$,  одержуємо
\begin{align}\label{Lemma1Estim4}
&\sum\limits_{k=1}^{\infty}\sum\limits_{v=(2k+1)n-k}^{\infty}e^{-\alpha v^{r}}
<
\frac{14}{13}\frac{1}{\alpha r} \sum\limits_{k=1}^{\infty}
 e^{-\alpha ( (2k+1)n-k )^{r}} (2k+1)n-k )^{1-r} \notag \\
 <&
 \frac{14}{13}\frac{1}{\alpha r}
 \left( e^{-\alpha(3n-1)^{r}}(3n-1)^{1-r}+
\int\limits_{1}^{\infty} e^{-\alpha ((2n-1)u+n)^{r}}
 (2n-1)u+n )^{1-r} du
  \right)
  \notag \\
= & 
\frac{14}{13}\frac{1}{\alpha r}
 \left( e^{-\alpha(3n-1)^{r}}(3n-1)^{1-r}+
 \frac{1}{2n-1}\int\limits_{3n-1}^{\infty} e^{-\alpha t^{r}}
t^{1-r} dt
  \right)
   \notag \\
\leq & 
\frac{14}{13}\frac{1}{\alpha r}
 \left( e^{-\alpha(3n-1)^{r}}(3n-1)^{1-r}+
 \frac{1}{2n-1} \frac{e^{-\alpha(3n-1)^{r}} }{\alpha r}
 (3n-1)^{2(1-r)}  \left(1+ \frac{14}{13}\frac{2(1-r)}{\alpha r(3n-1)^{r}} \right) 
  \right)
   \notag \\
< & 
\frac{14}{13}\frac{1}{\alpha r}
 e^{-\alpha(3n-1)^{r}}(3n-1)^{1-r}
 \left( 
1+\frac{(3n-1)^{1-r}}{(2n-1)\alpha r} \left(1+ \frac{14}{13}\frac{2}{\alpha r(3n-1)^{r}} \right) 
 \right) 
  \notag \\
= & 
\frac{14}{13}\frac{1}{\alpha r}
 e^{-\alpha(3n-1)^{r}}(3n-1)^{1-r}
 \left( 
1+\frac{3n-1}{2n-1} \frac{1}{\alpha r (3n-1)^{r}} \left(1+ \frac{28}{13}\frac{1}{\alpha r(3n-1)^{r}} \right) 
 \right) 
  \notag \\
< & 
\frac{14}{13}\frac{1}{\alpha r}
 e^{-\alpha(3n-1)^{r}}(3n-1)^{1-r}
 \left( 
1+2\cdot\frac{1}{14}   \left( 1+2\cdot\frac{14}{13}\cdot\frac{1}{14}\right)\right) 
  \notag \\
= &\frac{212}{169} e^{-\alpha(3n-1)^{r}}  \frac{(3n-1)^{1-r}}{\alpha r}<
\frac{636}{169}\frac{n^{1-r}}{\alpha r}  e^{-\alpha(3n-1)^{r}}.
\end{align}

Лему~\ref{Lemma1} доведено.
\end{proof}

Дана робота частково підтримана грантом H2020-MSCA-RISE-2019, номер проєкту 873071 (SOMPATY: Spectral Optimization: From Mathematics to Physics and Advanced Technology), а також фондом Фольсквагена (VolkswagenStiftung),  програмою “From Modeling and Analysis to Approximation”.



\vspace{1cm} 
\begin{thebibliography}{99}

\Large


\bibitem{Kolmogorov1935}
 A. Kolmogoroff, {\it Zur Gr\"{o}ssennordnung des Restgliedes
Fourierschen Reihen differenzierbarer Funktionen},  Ann. Math.(2),
{\bf 36},  №2,  521--526 (1935).


\bibitem{Korn}
{Н.П. Корнейчук}, {\it Точные константы в
теории приближения},   Наука, Москва,
(1987).



\bibitem{Nikolsky1945}
 С.М.  Никольский, {\it  Приближение периодических функций тригонометрическими полиномами}, Тр. Мат. ин-та
АН СССР, 15, (1945).


\bibitem{Oskolkov}
{\it  K.I. Oskolkov}, Inequalities of the "large size" type and applicatiojns to problems of trigonometric approximation, Anal. Math. {\bf 12},  143--166 (1986). 


\bibitem {SerdyukDopov1999}
А.С.  Сердюк, {\it Про асимптотично точні оцінки похибки наближення інтерполяційними тригонометричними поліномами функцій високої гладкості},  Доп. НАН України,  No.8, 29-33 (1999).


 


\bibitem{Serdyuk2012}
А.С.  Сердюк, {\it Наближення інтерполяційними тригонометричними поліномами на класах періодичних аналітичних функцій}, Укр. мат. журн.,  {\bf 64}, №5,  698--712 (2012). 







\bibitem{Serdyuk2004}
А.С.  Сердюк, {\it Наближення нескінченно диференційовних періодичних функцій інтерполяційними тригонометричними поліномами},  Укр. мат. журн., {\bf 56}, №4,   495--505 (2004).




\bibitem{SerdyukVoitovych2010}
А.С. Сердюк, В.А. Войтович, {\it   Наближення класів цілих функцій інтерполяційними аналогами сум Валле Пуссена},  Збірник праць Інституту математики НАН України, {\bf 7}, № 1: Теорія наближення функцій та суміжні питання.- Київ: Ін-т математики НАН України, 274-297 (2010).


\bibitem{SerdyukSokolenko2016}
А.С. Сердюк, І.В. Соколенко,  {\it  Наближення класів класів $(\psi,\beta)$-диференційовних функцій інтерполяційними тригонометричними поліномами}, Зб. Праць Ін-ту математики НАН України, {\bf  13}, No. 1, 289--299 (2016).

\bibitem{SerdyukSokolenko2017}
А.С. Сердюк, І.В. Соколенко,  {\it  Апроксимацiя класiв згорток перiодичних функцiй лiнiйними методами, побудованимина основi коефiцiєнтiв Фур’є–Лагранжа},  Зб. Праць Ін-ту математики НАН України, {\bf  14}, No. 1, 238--248 (2017).

\bibitem{SerdyukSokolenko2019}
А.С. Сердюк, І.В. Соколенко,  {\it   Наближення інтерполяційними тригонометричними поліномами в метриках просторів $L_{p}$ на класах періодичних цілих функцій},  Укр. мат. журн., {\bf 71}, No 2,    283 -- 292 (2019). 


\bibitem{SerdyukSokolenko2022}
А.С. Сердюк, І.В. Соколенко,  {\it  Наближення сумами Фур’є на класах диференційовних у сенсі Вейля--Надя функцій із високим показником гладкості}
Укр. мат. журн., {\bf 74},  №5,   685--700 (2022). 




\bibitem{SerdyukStepanyukDopov}
А.С. Сердюк, Т.А. Степанюк, {\it   Рівномірні наближення сумами Фур’є на класах згорток з інтегралами Пуассона}, Допов. НАН України,  № 11, 10--16 (2016). 

\bibitem{SerdyukStepanyuk2018}
 A.S. Serdyuk, T.A. Stepanyuk, {\it Uniform approximations by Fourier sums on  classes of generalized Poisson integrals, Analysis Mathematica}, {\bf 45}, №1,   201--236 (2019).

\bibitem{SerdyukStepanyuk2017}
А.С. Сердюк, Т.А. Степанюк, {\it  Наближення класів узагальнених інтегралів Пуассона сумами Фур’є в метриках просторів},
Укр. мат. журн., {\bf 69}, №5,   695--704 (2017).



\bibitem{SerdyukStepanyukFilomat} 
A.S. Serdyuk, T.A. Stepanyuk, {\it Asymptotically best possible Lebesque-type inequalities for the Fourier sums on sets of generalized Poisson integrals}, FILOMAT,  {\bf 34}, №14, 4697--4707  (2020).

\bibitem{SerdyukStepanyukJAEN}
A.S. Serdyuk, T.A. Stepanyuk, {\it About Lebesgue inequalities on the classes of generalized Poisson integrals}, Jaen J. Approx. {\bf 12}, 25--40 (2021).

\bibitem{SerdyukStepanyuk_Interpolation2023}
A.S. Serdyuk, T.A. Stepanyuk, {\it Uniform approximations by Fourier sums on the sets of convolutions of periodic functions of high smoothness}, arXiv:2301.02017v1 (2023).

\bibitem{Stepanets1}
{ А.И. Степанец, \/} {\it Методы теории
приближений}: В 2 ч.,  Пр. Iн-ту математики НАН України, Ин-т
математики НАН Украины, Київ, {\bf 40}, Ч. І 
(2002).

\bibitem{Stepanets2}
{ А.И. Степанец, \/} {\it Методы теории
приближений}: В 2 ч.,  Пр. Iн-ту математики НАН України, Ин-т
математики НАН Украины, Київ, {\bf 40}, Ч. ІІ
(2002).

\bibitem{StepanetsSerdyuk2000}
А.И. Степанец,  А.С. Сердюк, { Приближение периодических аналитических функций интерполяционными тригонометрическими многочленами}, Укр. мат. журн., {\bf 59}, №12, 1689--1701 (2000).

  
\bibitem{StepanetsSerdyuk2000Zb}
О.І. Степанець, А.С. Сердюк, {\it   Оцінка залишку наближення інтерполяційними тригонометричними многочленами на класах нескінченно диференційовних функцій},  Теорія наближення функцій та її застосування : Праці Ін-ту математики НАН України, {\bf 31},  446--460 (2000).


\bibitem{Stepanets_Serdyuk_Shydlich2007}
О.І. Степанець, А.С. Сердюк, А.Л. Шидліч, {\it Про деякі нові критерії нескінченної диференційовності періодичних функцій}, Укр. мат. журн., {\bf 59}, №10, 1399--1409 (2007).


\bibitem {Stepanets_Serdyuk_Shydlich2009}
А.И. Степанец, А.С. Сердюк, А.Л. Шидлич, {\it О связи классов $(\psi, \overline{\beta})$-дифференцируемых функций с классами Жевре}, Укр. мат. журн.,
 {\bf 61}, №1,  140--145 (2009).
 
 \bibitem{Stechkin1980}
С.Б. Стечкин, {\it Оценка остатка ряда Фурье для дифференцируемых функций}, Приближение функций полиномами и сплайнами, Сборник статей, Тр. МИАН СССР, 145,  126--151 (1980). 
 



\bibitem{Teljakovsky1989}
С.А. Теляковский,  {\it  О приближении суммами Фурье функций высокой
гладкости } Укр. мат. журн.   {\bf 41}, №41,  510--518 (1989).


\bibitem{shakirov2011}
И.А. Шакиров, {\it  Полное исследование функций Лебега, соответствующих классическим интерполяционным полиномам Лагранжа}, Изв. вузов. Матем., No 10, 80--88 (2011).


\bibitem{shakirov2018}
{ И.А. Шакиров}, {\it   Приближение константы Лебега полинома Лагранжа логарифмической функцией со смещением аргумента}, Итоги науки и техн. Сер. Соврем. мат. и ее прил. Темат. обз., 153,  151--157 (2018).



\bibitem{Sharapudinov}
{ I.I. Sharapudinov},
{\it  On the best approximation and polynomials of the least quadratic deviation}, Analysis Mathematica, {\bf 9},  223--234 (1983).













\end{thebibliography}
\end{document}